\documentclass[12pt,a4paper]{article}
\usepackage[utf8]{inputenc}
\usepackage{lmodern}
\usepackage[T1]{fontenc}
\usepackage[english]{babel}
\usepackage{ifpdf}
\usepackage{tipa}
\usepackage[usenames,dvipsnames]{xcolor}
\usepackage{graphicx}
\usepackage[shortlabels]{enumitem}
\usepackage[numbers,sort&compress]{natbib}
\usepackage{doi}
\usepackage[font=footnotesize,width=\textwidth]{caption}
\usepackage{booktabs}
\usepackage{hyperref}
\usepackage[top=2.0cm, bottom=3.0cm]{geometry}
\usepackage[affil-it]{authblk}
\usepackage[parfill]{parskip}
\usepackage[setbuildercolon]{matmatix}
\usepackage{color}

\newcommand{\revision}{}
\ifpdf
  \hypersetup{
    pdfauthor={François Bienvenu}
    pdftitle={Combinatorial and stochastic properties of ranked tree-child networks}
  }
\fi

\hypersetup{colorlinks, allcolors = {MidnightBlue}}

\newenvironment{mathlist}
{\begin{enumerate}[label={\upshape(\roman*)}, align=left, widest=iii, leftmargin=*]}
{\end{enumerate}\ignorespacesafterend}

\makeatletter
\newtheorem*{rep@theorem}{\rep@title}
\newcommand{\newreptheorem}[2]{%
\newenvironment{rep#1}[1]{%
 \def\rep@title{#2 \ref*{##1}}%
 \begin{rep@theorem}}%
 {\end{rep@theorem}}}
\makeatother

\newtheorem{theorem}{Theorem}[section]
\newreptheorem{theorem}{Theorem}
\newtheorem*{theorem*}{Theorem}
\newtheorem{quotedtheorem}{Theorem}

\newtheorem{proposition}[theorem]{Proposition}
\newtheorem{lemma}[theorem]{Lemma}
\newreptheorem{lemma}{Lemma}
\newreptheorem{proposition}{Proposition}
\newtheorem{corollary}[theorem]{Corollary}
\newtheorem{conjecture}[theorem]{Conjecture}
\newreptheorem{conjecture}{Conjecture}

\theoremstyle{definition}

\makeatletter
\newtheorem*{rep@definition}{\rep@title}
\newcommand{\newrepdefinition}[2]{%
\newenvironment{rep#1}[1]{%
 \def\rep@title{#2 \ref*{##1}}%
 \begin{rep@definition}}%
 {\end{rep@definition}}}
\makeatother

\newtheorem{definitionX}[theorem]{Definition}
\newenvironment{definition}
  {\pushQED{\qed}\definitionX}
  {\popQED\enddefinitionX}
\newrepdefinition{definition}{Definition}
\newtheorem*{notationX}{Notation}
\newenvironment{notation}
  {\pushQED{\qed}\notationX}
  {\popQED\endnotationX}
\newtheorem{remarkX}[theorem]{Remark}
\newenvironment{remark}
  {\pushQED{\qed}\remarkX}
  {\popQED\enddefinitionX}

\makeatletter
\providecommand{\reachedby}{%
  \mathrel{\mathpalette\reflect@squig\relax}%
}
\newcommand{\reflect@squig}[2]{%
  \reflectbox{$\m@th#1\rightsquigarrow$}%
}
\makeatother

\newcommand{\fallfact}[2]{#1^{\underline{#2}}}

\newcommand{\coevent}{\,\mathcal{R}\,}
\newcommand{\gammaDown}{\gamma^{\downarrow}}
\newcommand{\gammaUp}{\gamma^{\uparrow}}
\newcommand{\eulergamma}{\text{\textgamma}}
\newcommand\ring[1]{\mathaccent23{#1}}
\def\Vr{\ring{V}} % {\mbox{\r{\itshape V}}

\title{Combinatorial and stochastic properties of ranked tree-child networks}

\begin{document}

\author[1,2]{François Bienvenu}
\author[1,2]{Amaury Lambert}
\author[3]{Mike Steel}

\affil[1]{\small Center for Interdisciplinary Research in Biology (CIRB),
CNRS UMR 7241,\newline Collège de France, PSL Research University, 
Paris, France}
\affil[2]{Laboratoire de Probabilités, Statistique et Modélisation (LPSM),
CNRS UMR 8001, Sorbonne Université, Paris, France}
\affil[3]{Biomathematics Research Centre, University of Canterbury,
Christchurch,\newline New Zealand}

\maketitle

\begin{abstract}
Tree-child networks are a recently-described class of directed acyclic
graphs that have risen to prominence in phylogenetics (the study of evolutionary
trees and networks). Although these networks have a number of attractive
mathematical properties, many combinatorial questions concerning them remain
intractable. In this paper, we show that endowing these networks with
a biologically relevant ranking structure yields mathematically
tractable objects, which we term ranked tree-child networks (RTCNs).
We explain how to derive exact and explicit combinatorial results concerning the
enumeration and generation of these networks. We also explore
probabilistic questions concerning the properties of RTCNs
when they are sampled uniformly at random. These questions
include the lengths of random walks between the root and leaves (both
from the root to the leaves and from a leaf to the root); the
distribution of the number of cherries in the network; and sampling RTCNs
conditional on displaying a given tree. We also formulate a conjecture
regarding the scaling limit of the process that counts
the number of lineages in the ancestry of a leaf.
The main idea in this paper, namely using ranking as a way to achieve
combinatorial tractability, may also extend to other classes of
networks.
\end{abstract}

%\tableofcontents

\section{Introduction} \label{secRTCNintro}

Tree-child networks are a class of directed acyclic graphs (DAGs) introduced
by~\cite{Cardona2009TCBB} as a way to model reticulated phylogenies (that is,
phylogenies that take into account the possibility of hybridization or
horizontal gene transfer). In addition to being biologically relevant,
tree-child networks are mathematically interesting combinatorial structures and
have thus gained attention recently to become one of the most studied classes
of phylogenetic networks. Although they are simpler and more structured than
arbitrary phylogenetic networks, they are nevertheless notoriously hard to study.
For instance, their enumeration is still an open
problem~\cite{McDiarmid2015AnnComb, FuchsArXiv} and there is no known algorithm
to sample them uniformly (although recursive procedures to enumerate them
have recently been introduced \cite{CardonaArXiv, CardonaZhang}, 
as well as asymptotic methods \cite{Fuchs2ArXiv}).  As a result, relatively little is
known about the properties of ``typical'' tree-child networks.

In this paper, we introduce a new class of phylogenetic networks that
we term \emph{ranked tree-child networks}, or RTCNs for short. 
These networks correspond to a subclass of tree-child networks that are endowed
with an additional structure ensuring that they could have resulted from a
time-embedded evolutionary process, something that is not required of
tree-child networks.

Besides being arguably more biologically relevant than tree-child networks,
one of the main advantages of RTCNs is that they are much easier to study.
For instance, we show that there exist explicit formulas for the number
of leaf-labeled RTCNs, as well as simple procedures to sample them
uniformly at random (or even uniformly at random subject to some natural
constraints such as containing a fixed number of reticulations, or displaying a
given tree). These features make it possible to gain some insight into the
structure of uniform RTCNs.

The structure of this paper is as follows.  We begin by describing the class of
of tree-child networks, and then define the new class of RTCNs.  We then collect together
statements of the main results of this paper.  In Section~\ref{secRTCNEnum}, we
use forward-in-time and backward-in-time constructions of RTCNs to derive exact
results for their enumeration and generation. In
Section~\ref{secRTCNLinkTrees}, we describe the probability distribution for
the number of ``cherries'' and ``tridents'' in a RTCN that has been sampled
uniformly at random.  Section~\ref{secRTCNRandPaths} provides an analysis of
the lengths of random walks in a uniform RTCN, both from the root to
a leaf, and from a leaf to the root. Finally, in
Section~\ref{secRTCNAncestryLeaf} we investigate the number of lineages in the
ancestry of a leaf. We characterize the Markov chain that counts these lineages as
we move away from the leaf, and conjecture that when appropriately rescaled
it converges, as the number of leaves of the RTCN goes to infinity, to a smooth
function with a random shape parameter.

\subsection{Preliminaries} \label{secRTCNprelim}

Let us start by recalling the definition of tree-child networks and introducing
some vocabulary.

\begin{definition} \label{defBinRetic}
A \emph{binary phylogenetic network} is a connected directed acyclic graph
which has a single vertex of
in-degree 0 and out-degree 2 (the {\em root}) and where
every other vertex has either:
\begin{itemize}
  \item in-degree 1 and out-degree 0 (the \emph{leaves})
  \item in-degree 1 and out-degree 2 (the \emph{tree vertices})
  \item in-degree 2 and out-degree 1 (the \emph{reticulation vertices}.) \qedhere
\end{itemize}
\end{definition}

\enlargethispage{1ex}
If $V$ is the vertex set of a binary phylogenetic network, we write
$\partial V$ for the set of its leaves.
The vertices that are not leaves are called \emph{internal vertices} and we
denote their set by $\Vr$.

We refer to the elements of the set
$\Gamma_{\!\text{in}}(v) = \Set{u \suchthat u \to v}$ as the \emph{parents} of
$v$ and to that of the set
$\Gamma_{\!\text{out}}(v) = \Set{u \suchthat v \to u}$ as the \emph{children}
of~$v$. 

%Two vertices are said to be \emph{siblings} if they share a
%parent and \emph{step-siblings} if they share a sibling.

Finally, an edge $\vec{uv}$ is called a \emph{reticulation edge} if
$v$ is a reticulation vertex and a \emph{tree edge} if $v$ is a tree vertex
or a leaf.

\begin{definition} \label{defTCN}
A \emph{tree-child network} is a binary phylogenetic network such
that every internal vertex has at least one child that is a tree vertex or a
leaf.
\end{definition}

Note that there are other simple characterizations of tree-child networks.
For example, Lemma~2 in \cite{Cardona2009TCBB} gives the following alternative
definition: a binary phylogenetic network is tree-child if and only if for
every vertex $v$ there exists a leaf such that every path going from the root
to that leaf goes through~$v$. \revision{See Figure~\ref{figRTCN00} for
examples of tree-child networks and of a binary phylogenetic network that
is not a tree-child network.}

\begin{figure}[h!]
  \centering
  \captionsetup{width=0.85\linewidth}
  \includegraphics[width=0.85\linewidth]{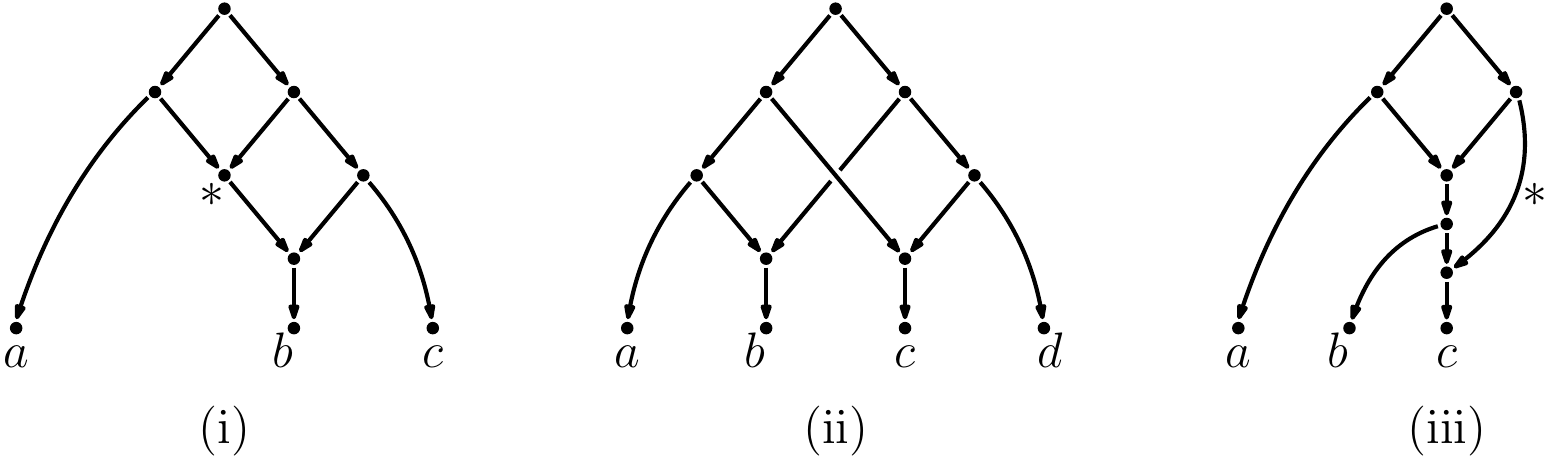}
  \caption{\revision{(i) A binary phylogenetic network that fails to be tree-child,
  since the only path from the vertex indicated by ``\textasteriskcentered'' to
  the leaves passes through a reticulation vertex; (ii) A tree-child network
  that cannot be ranked (i.e.\ does not correspond to any RTCN); (iii) A
  tree-child network that is not a normal network, as the arc indicated by
  ``\textasteriskcentered'' is redundant
  (see Definition~\ref{defNormalNetwork}).}}  
\label{figRTCN00}
\end{figure}

\subsection{Ranked tree-child networks} \label{secRTCNDef}

First, note that every every tree-child network is endowed with
a partial order, which we refer to as the \emph{genealogical order}, defined by
\[
  u \leadsto v \iff
  \begin{cases}
    \text{there exists a directed path from $u$ to $v$} \\[-0.5ex]
  \text{that contains at least one tree-edge.}
  \end{cases}
\]
Let us now introduce the  notion of \emph{events} of a tree-child network.

\begin{definition}
Let $N$ be a tree-child network. Define a relation $\mathcal{R}$
on the set~$\Vr$ of internal vertices of $N$ by:
\[
  u \coevent v \iff
  \text{$u=v$ or $u$ and $v$ are linked by a reticulation edge or share a child}\,.
\]
It is easy to see that, for tree-child networks, $\mathcal{R}$ is an
equivalence relation on $\Vr$ (note that this is not the case for general
networks). We call the equivalence classes of $\mathcal{R}$ the \emph{events}
of $N$.
Moreover, writing $\bar{u}$ for the equivalence class of a vertex $u$,
\begin{itemize}
  \item either $\bar{u} = \Set{u}$, in which case $\bar{u}$ is called a
    \emph{branching event};
  \item or \revision{$\bar{u}$ has three elements}, and $\bar{u}$ is called a
    \emph{reticulation event}.
\end{itemize}
\revision{When there is no risk of confusion, we refer to these events as
\emph{branchings} and \emph{reticulations}, for short.}
\end{definition}

\pagebreak

\begin{definition} \label{defRTCN}
  A \emph{ranked tree-child network} is an ordered pair $(N, \prec)$ where
  \begin{itemize}
    \item $N$ is a tree-child network.
    \item The \emph{chronological order} $\prec$ is a strict total order on the
      set of events of $N$ that is compatible with the genealogical order;
      in other words, for every internal vertices $u$ and $v$,
      \[
        u \leadsto v \;\implies\;
        \bar{u} \prec \bar{v} .
      \]
  \end{itemize}
\revision{See Figure~\ref{figRTCN01} for a graphical representation of a ranked
  tree-child network.}
\end{definition}
Note that in the case where $N$ is a tree, Definition~\ref{defRTCN} agrees
with the classical notion of ranked tree (see e.g.\ \cite{LambertBrazJProbabStat2017}). Indeed, in that case every
internal vertex is its own equivalence class and thus the chronological order can
be seen as a total strict order on $\Vr$.

The idea behind the notion of ranking is that the internal vertices of
a phylogeny can be associated to evolutionary events. Under the assumption
that these events are instantaneous and that no two events can happen at the
same time, ranking a phylogeny is a straightforward way to endow it with the
information of the order in which these events occurred,
without having to specify the actual times at which they did.

\begin{figure}[h!]
  \centering
  \captionsetup{width=0.85\linewidth}
  \includegraphics[width=0.85\linewidth]{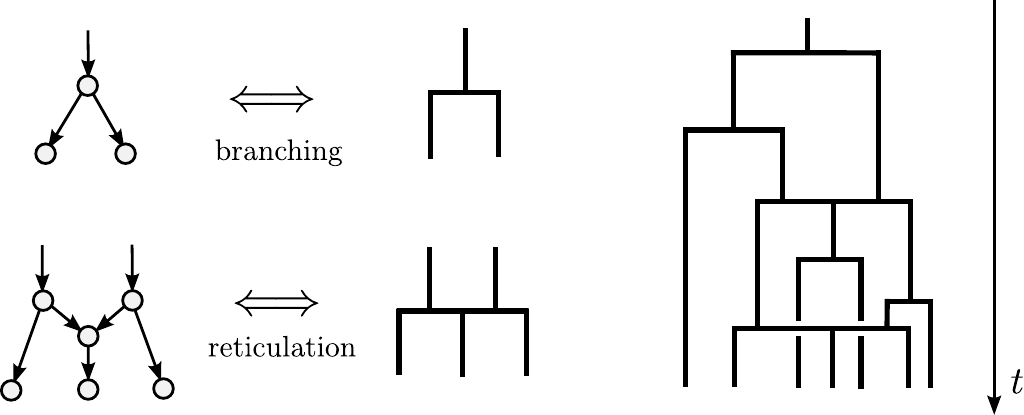}
  \caption{Graphical representation of a RTCN. 
  Each ``$\boldsymbol\bot$'' corresponds to a tree vertex (or the root), each
  ``$\boldsymbol\top$'' to a reticulation vertex,
  and the tip of each dangling vertical line to a leaf.
  The vertical lines represent tree edges, and
  the horizontal ones events. Note that
  $a \prec b$ if and only if the horizontal line representing $b$
  is below that of $a$.}  
\label{figRTCN01}
\end{figure}

The notion of ranking raises two questions:
\begin{enumerate}
  \item Can all tree-child networks be ranked?
  \item How many rankings are there for a given tree-child network?
\end{enumerate}

It is not hard to see that the answer to the first question is no.
In fact, we will prove in the next section
that almost no tree-child network can be ranked. While this rules out the
possibility of using RTCNs to gain insight into the structure of tree-child
networks, this does not make RTCNs irrelevant, because the tree-child networks
that we expect to see in nature should typically be rankable\,\footnote{\,This
is not necessarily the case if some extinct lineages are not observed; but in
that case we can always assume that there exists an underlying ranked
tree-child network.}. Thus, from a biological point of view, RTCNs
might actually be more relevant than tree-child networks.

Regarding the second question, \revision{it is well-known
(see e.g.\ Chapter 5, Exercice~20 of~\cite{Knuth1997} or Lemma~1
in~\cite{steel2001properties}) that
the number of ways to rank a binary tree is}
\[
  \frac{(\ell - 1)!}{\prod_{v \in \Vr}(\lambda(v) - 1)}\,, 
\]
where $\ell$ is the number of leaves of the tree
and $\lambda(v)$ the number of leaves
subtended by $v$ -- that is, ${\lambda(v) = \# \Set{u \in \partial V \suchthat v
\leadsto u}}$. 
However, the proof of this result relies on the recursive
structure of trees, which tree-child networks lack. Thus, counting the number
of ways to rank a given tree-child network remains an open question.

\subsection{Relation to other types of phylogenetic networks}

Let us start by recalling the definition of temporal networks.

\begin{definition} \label{defTemporal}
A binary phylogenetic network $N = (V, \vec{E})$ is \emph{temporal}
if there exists a (time-stamp) function $t: V \rightarrow \R$ that satisfies the following two properties:
\begin{mathlist}
\item If $\vec{uv}$ is a reticulation edge then $t(u) = t(v)$.
\item If $\vec{uv}$ is a tree edge then $t(u) < t(v)$. \qedhere
\end{mathlist}
\end{definition}

The underlying tree-child network of a RTCN is always a temporal network, because
to get a valid time-stamp function~$t$ one can assign to every internal vertex
the rank of the corresponding event (that is, $t(u) = k$ where $\bar{u}$ is the
$k$-th event, and $t(\rho) = 0$ for the root). Conversely, every temporal
tree-child network can be ranked, because one can tweak the time-stamps to make
sure that all reticulation vertices have different time-stamps,  and then set
\[
  \bar{u} = \bar{v} \iff t(u) = t(v)\quad\text{and}\quad
  \bar{u} \prec \bar{v} \iff t(u) < t(v)
\]
to get a valid partition into events and a valid ranking.
This gives us the following proposition.

\begin{proposition}
A tree-child network is temporal if and only if it is the underlying network
of a RTCN.
\end{proposition}

Although temporal tree-child networks and RTCNs are closely related,
the two should not be mistaken:
a RTCN contains more information
than the corresponding temporal tree-child network, and a slightly
different information than
this network together with its time-stamp function (note that there is no
requirement that the time-stamps of vertices that belong to different events be
distinct in Definition~\ref{defTemporal}). RTCNs should be thought of
as being to temporal tree-child networks what ranked trees are to trees.

Let us now recall the notion of normal network, as introduced
by~\cite{Willson2007BMBa}.

\begin{definition} \label{defNormalNetwork}
An edge $\vec{uv}$ is said to be \emph{redundant} (or a \emph{shortcut}) if
there exists a directed path from $u$ to $v$ that does not contain $\vec{uv}$
\revision{(see Figure~\ref{figRTCN00})}.
A \emph{normal network} is a tree-child network that has no redundant edges.
\end{definition}

It is well-known and not too hard to see that every temporal tree-child network
is a normal network (see e.g.\ Proposition~10.12 in \cite{Steel2016}). As a
result, rankable tree-child networks are normal networks, in the sense that if
$(N, \prec)$ is a RTCN then $N$ is a normal network.
Since by Theorem~1.4 in \cite{McDiarmid2015AnnComb} the fraction of
tree-child networks (leaf-labeled or vertex-labeled alike) that are also
normal networks goes to zero as their number of vertices goes to infinity,
this proves the next proposition.

\begin{proposition} \label{propRTCNPropRankable}
\revision{Let $\tilde{C}_n$ be the number of tree-child networks with
$n$ labeled vertices, and let $\tilde{R}_n$ be the number of those
tree-child networks that are rankable. Then, as $n$ goes to infinity,
$\tilde{R}_n / \tilde{C}_n \to 0$.}

\revision{Similarly, the fraction $R_\ell / C_\ell$ of leaf-labeled tree-child
networks that are rankable goes to 0 as $\ell$ goes to infinity.}
\end{proposition}

\subsection{Main results} \label{secRTCNMainResults}

All the results presented in this paper concern leaf-labeled RTCNs \revision{--
that is, the leaves carry labels that uniquely identify them,
but the internal vertices do not.}
Note that the number of reticulations of a RTCN with
$\ell$ leaves and $b$ branchings is $r = \ell - b - 1$.

In Section~\ref{secRTCNEnum}, we give two constructions of RTCNs: 
one in backward time and one in forward time.
Each of these constructions yields a proof of the following result.

\begin{reptheorem}{thmEnumRTCN}
The number of ranked tree-child networks with
$\ell$ labeled leaves and $b$~branchings is
\[
  C_{\ell, b} \;=\;
  {\ell - 1 \brack b}\, T_\ell\,,
\]
where $T_\ell = \ell!\,(\ell - 1)! /\, 2^{\ell - 1}$\! is the number of ranked
trees with $\ell$ labeled leaves and~${\ell - 1 \brack b}$ is the number of
permutations of \,$\Set{1, \ldots, \ell - 1}$ with $b$ cycles
(these quantities are known as the unsigned Stirling numbers of the first
kind).
\end{reptheorem}

\revision{In particular, Theorem~\ref{thmEnumRTCN} implies that the number of
branchings of a uniform RTCN with $\ell$ labeled leaves is asymptotically
Poisson with mean $\sim\log(\ell)$, and thus satisfies a central limit theorem (see
Corollary~\ref{propLawProfile})}.

The backward-time and forward-time constructions also
provide simple procedures to sample leaf-labeled RTCNs, be it:
\begin{itemize}
  \item uniformly at random;
  \item uniformly at random conditional on their total number of reticulations;
  \item uniformly at random conditional on which events are reticulations.
\end{itemize}

\revision{The rest of our study focuses on the properties of the uniform
distribution on the set of RTCNs with $\ell$ labeled leaves.} One of
the interesting characteristics of uniform leaf-labeled RTCNs
is their close relationship with uniform
leaf-labeled ranked trees. This is detailed in Section~\ref{secRTCNLinkTrees},
where we also explain how to sample a uniform RTCN conditional on displaying a
given tree \revision{(i.e.\ on ``containing'' the tree, in the sense of
Definition~\ref{defDisplay})}.

Two of the most basic statistics of binary phylogenetic networks are their
numbers of ``cherries'' and of ``tridents'' (we will define these
precisely later, but briefly,  a cherry is a vertex having two leaves as
children; a trident is a reticulation event with three leaves as children).
While almost nothing is known about the distribution of these two
quantities in uniform tree-child networks, they prove very tractable in
uniform RTCNs.  Explicit expressions for their mean and variance are given in
Section~\ref{secRTCNStats}, where we also establish the following theorem.

\begin{reptheorem}{thmRTCNCherriesRetCherries}
Let $\kappa_\ell$ be the number of cherries of a uniform RTCN with
$\ell$ labeled leaves and $\chi_\ell$ be its number of tridents.
Then, as $\ell \to \infty$,
\begin{mathlist}
  \item $\displaystyle \kappa_\ell \;\tendsto[\,d\,]{}\;
    \mathrm{Poisson}\big(\tfrac{1}{4}\big)$.
  \item $\displaystyle \chi_\ell \,/\, \ell \;\tendsto[\,\P\,]{}\; \tfrac{1}{7}$.
\end{mathlist}
\end{reptheorem}

Sections~\ref{secRTCNRandPaths}
and~\ref{secRTCNAncestryLeaf}
contain our most informative results about the structure of uniform RTCNs:
in Section~\ref{secRTCNRandPaths} we study the length of some typical paths
joining the root to the leaf set of uniform RTCNs and prove the following
theorem.

\begin{reptheorem}{thmRTCNRandPaths}
Let $\nu$ be a uniform RTCN with $\ell$ labeled leaves and let
\begin{itemize}
  \item \revision{$\gammaDown$ be the path taken by a random walk going from the
    root of $\nu$ to its leaves, respecting the direction of the edges and
    choosing each of the two outgoing edges with equal probability when it
    reaches a tree-vertex.}
  \item \revision{$\gammaUp$ be the path taken by a random walk going from a
    uniformly chosen leaf of $\nu$ to its root, following the edges in reverse
    direction and choosing each of the two incoming edges with equal
    probability when it reaches a reticulation.}
\end{itemize}
Then, letting \,$\mathrm{length}(\,\cdot\,)$ denote the length of these paths, not
counting reticulation edges, there exist two constants $c^{_\downarrow}$ and
$c^{_\uparrow}$ such that, as $\ell\to \infty$,
\begin{mathlist}
\item $\displaystyle \mathrm{length}(\gammaDown) \;\approx\;
  \mathrm{Poisson}(2\log\ell + c^{_\downarrow})$  
  \item $\displaystyle \mathrm{length}(\gammaUp) \;\approx\;
    \mathrm{Poisson}(3\log\ell + c^{_\uparrow})$ 
\end{mathlist}
in the sense that the total variation distance between these distributions
goes to~0.
\end{reptheorem}

\begin{remark}
The proof of Theorem~\ref{thmRTCNRandPaths} can be adapted to the Yule model,
which corresponds to uniform ranked trees. In that case,
$\mathrm{length}(\gamma^\downarrow) \approx \mathrm{Poisson}(\log\ell+a)$ and
$\mathrm{length}(\gamma^\uparrow) \approx 
\mathrm{Poisson}(2 \log\ell+b)$, for suitable constants $a,b$.
\end{remark}

Finally, in Section~\ref{secRTCNAncestryLeaf} we study the
number of lineages in the ancestry of a leaf, i.e.\ in the subgraph
consisting of all paths joining this leaf to the root, as illustrated
in Figure~\ref{figRTCNDefAncestry} \revision{and defined formally in
Definition~\ref{defLineages}.} Starting from the leaves and going towards
the root, one event of the RTCN after the other, this number of lineages will
on average start by increasing but will eventually decrease as we get nearer
the root and the total number of lineages of the RTCN itself decreases.

\begin{figure}[h!]
  \centering
  \captionsetup{width=0.40\linewidth}
  \includegraphics[width=0.40\linewidth]{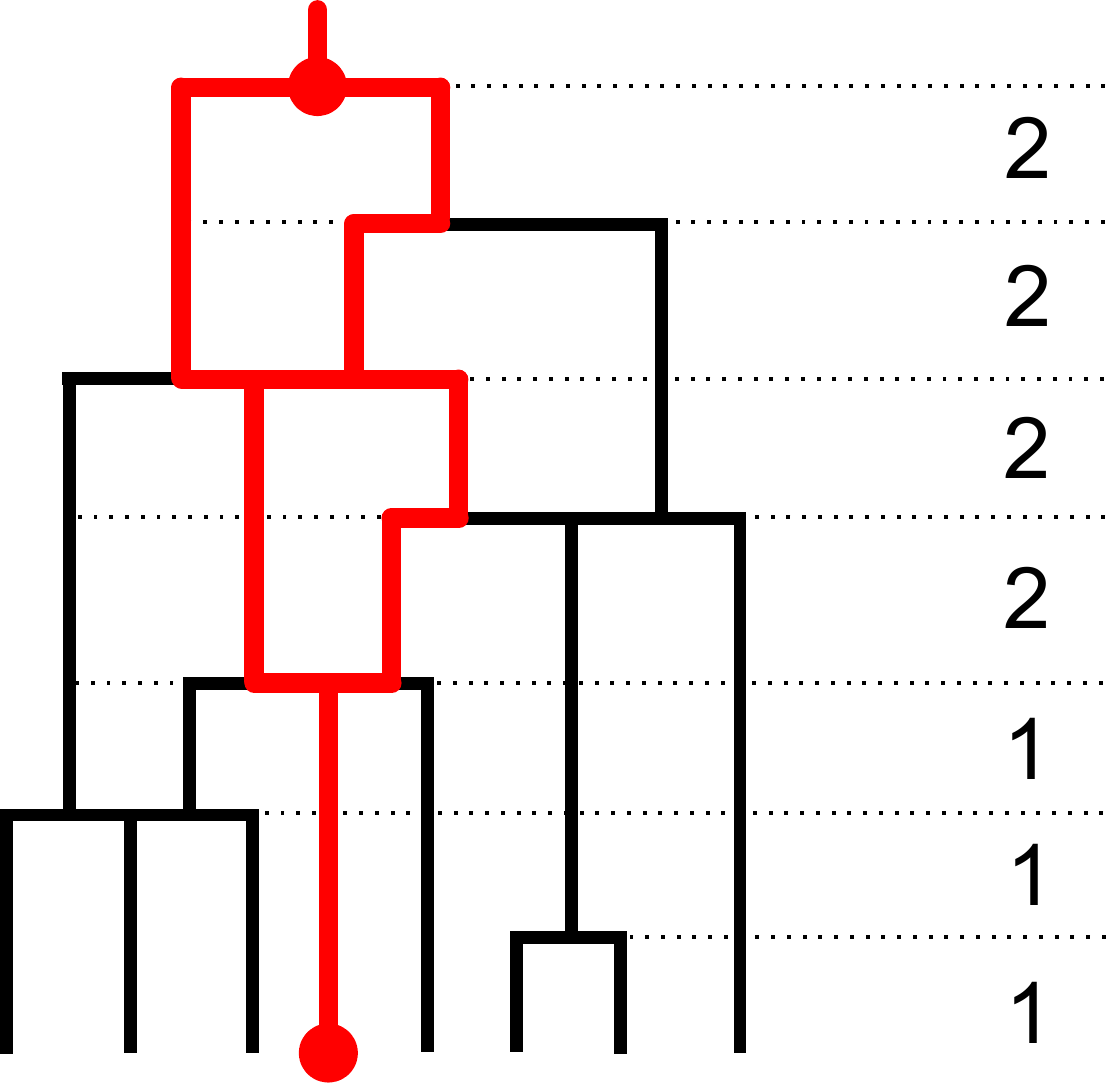}
  \caption{In red, the ancestry of a leaf and on the right the process
  counting the number of lineages in this ancestry.}  
\label{figRTCNDefAncestry}
\end{figure}

Let $X^{(\ell)} = (X_k^{(\ell)},\, 0 \leq k \leq \ell - 2)$
denote the process counting these lineages. What can we say about the
behavior of $X^{(\ell)}$ as $\ell$ goes to infinity?
As it turns out, $X_k^{(\ell)}$ remains quite small as long as
$k = \lfloor{x\ell}\rfloor$ with $0 < x < 1$. However, when we reach the regime
$k = \lfloor{\ell - x \sqrt{\ell}}\rfloor$, $X_k^{(\ell)}$ starts to increase
substantially, until a point where it will decrease suddenly.
Moreover, even though the
properly rescaled trajectories of $X^{(\ell)}$ seem to become smooth as $\ell
\to \infty$, they do not become deterministic. In fact, simulations presented
in Section~\ref{secRTCNSimus} suggest that they behave as a deterministic
process with a random initial condition.

Another natural way to study the asymptotics of
$X^{(\ell)}$, instead of restricting ourselves to
an appropriate time-window, is to consider the embedded process
$\tilde{X}^{(\ell)}$, that is, to ignore the steps in which $X^{(\ell)}$
does not change. As illustrated in Figure~\ref{figRTCNXtilde},
simulations suggest that the scaling limit of $\tilde{X}^{(\ell)}$ is a
remarkably regular deterministic process with a random shape parameter.

\begin{figure}[h!]
  \centering
  \captionsetup{width=\linewidth}
  \includegraphics[width=\linewidth]{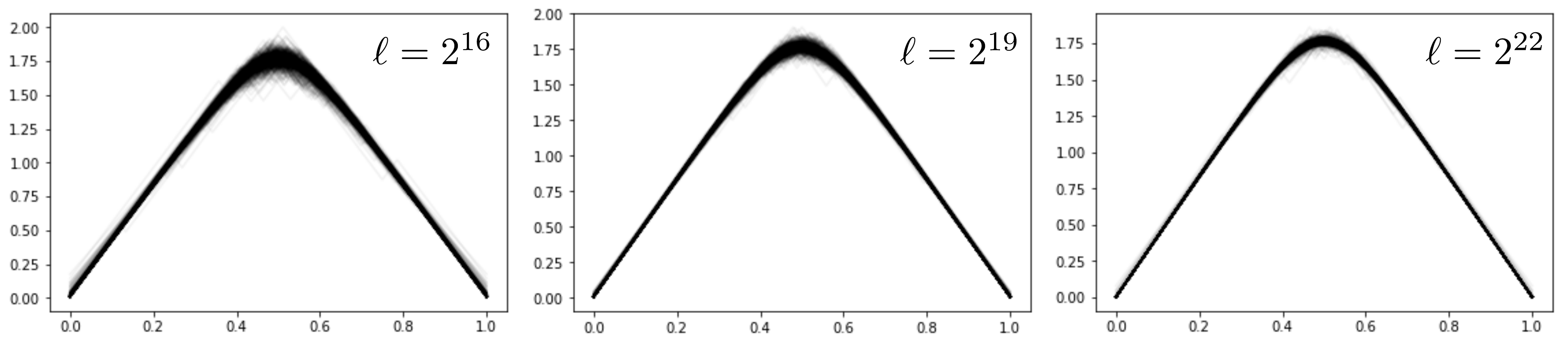}
  \caption{500 superimposed trajectories of the process
  $\tilde{X}^{(\ell)}$, rescaled so that their domain is $\ClosedInterval{0, 1}$
  and that the total area under each curve is equal to 1.}  
\label{figRTCNXtilde}
\end{figure}

We were not able to complete the proof of the convergence of these rescaled
processes, so our discussion relies in part on simulations and heuristics.
Below are some of our results and our main conjecture.

\begin{repproposition}{propRTCNExpecAncesLineage}
For all $\ell \geq 2$ and $k$ such that $0 \leq k \leq \ell - 2$,
\[
  \frac{\ell}{\ell - k + 1}
  \mleft(1 - \frac{2 k}{(\ell - k)(\ell - k - 1)}\mright) \;\leq\;
  \Expec{X_k^{(\ell)}} \;\leq\;
  \frac{\ell}{\ell - k}\ .
\]
As a result, for $M > 2$ and all $\ell$ large enough, for all $\epsilon > 0$,
\[
  (1 - \epsilon) \frac{1}{M} \mleft(1 - \frac{2}{M^2}\mright) \;\leq\;
  \Expec{\tfrac{1}{\sqrt{\ell}}
         X_{\lfloor\ell - M \sqrt{\ell}\rfloor}^{(\ell)}} \;\leq\;
  \frac{1}{M} \,, 
\]
so that the sequence of random variables
$\frac{1}{\sqrt{\ell}} X_{\lfloor\ell - M \sqrt{\ell}\rfloor}^{(\ell)}$
is tight and bounded away from 0 in $L^1$.
\end{repproposition}

\begin{repconjecture}{conjRTCNMain}
There exists a random variable $W_M > 0$ such that
\[
  \tfrac{1}{\sqrt{\ell}} X^{(\ell)}_{\lfloor\ell - M \sqrt{\ell}\rfloor}
  \tendsto[d]{\ell \to \infty} W_M\, .
\]
\end{repconjecture}

\begin{repproposition}{propRTCNConvergeDeter}
If Conjecture~\ref{conjRTCNMain} holds, then
for all $\epsilon$ such that $0 < \epsilon < 1$, as $\ell \to \infty$,
\[
  \Big(\tfrac{1}{M\sqrt{\ell}}
  X^{(\ell)}_{\lfloor \ell -M \sqrt{\ell}(1 - t)\rfloor}, \;
  t \in \ClosedInterval{0, 1 - \epsilon}\Big)
  \;\implies\;
  \big(y(t, C_M),\; t \in \ClosedInterval{0, 1 - \epsilon}\big)
\]
where $\implies$ denotes convergence in distribution in the Skorokhod space,
\[
  y(t, C_M) \;=\; \frac{1 - t}{C_M \cdot(1 - t)^2 + 1}
\]
and $C_M = M/W_M - 1$.
\end{repproposition}

%]]]

\section{Counting and generating RTCNs} \label{secRTCNEnum}
%[[[

\subsection{Backward-time construction of RTCNs}
\label{secRTCNBackwardInTimeEncoding}

Let us start by recalling that there is a simple way to
label the internal vertices of any leaf-labeled tree-child network (or any
leaf-labeled DAG), namely by labeling each internal vertex with the set of
labels of its children.

\begin{definition} Given a tree-child network with leaf set
  $\partial V = \Set{1, \ldots, \ell}$, the associated
  \emph{canonical labeling} is the function $\xi$ such that
  \begin{mathlist}
    \item $\forall v \in \partial V$, $\xi(v) = v$.
    \item $\forall v \in \Vr$, $\xi(v) = \Set{\xi(u) \suchthat v \to u}$.
    \qedhere
  \end{mathlist}
\end{definition}

While the canonical labeling of a tree-child network encodes it
unambiguously (the whole network can be recovered from the label of
the root), this is not the case for RTCNs because the information about the
order of the events is missing. One way to retain this information is to encode
RTCNs using the process $(P_k,\, 1 \leq k \leq \ell)$ defined as follows.

First, given a set $P = \Set{\xi_1, \ldots, \xi_m}$ of labels of vertices of
$(N, \prec)$, define the two operations:
\begin{itemize}
  \item $\displaystyle\mathrm{coal}(P, \Set{\xi_i, \xi_j}) = 
    \big(P \setminus \Set{\xi_i, \xi_j}\big) \cup
    \Set*[\big]{\Set{\xi_i, \xi_j}}$  
  \item $\displaystyle\mathrm{ret}(P, \xi_i, \Set{\xi_j, \xi_k}) = 
    \big(P \setminus \Set{\xi_i, \xi_j, \xi_k}\big) \cup
    \Set*[\big]{\Set{\{\xi_i\}, \xi_j},\, \Set{\{\xi_i\}, \xi_k}}$  
\end{itemize}

Now, let $U_1 \prec \cdots \prec U_{\ell - 1}$ denote the events of
$(N, \prec)$. Then, starting from
$P_1 = \Set{1, \ldots, \ell}$ and going backwards in time,
for $k = 1$ to $\ell - 1$:
\begin{itemize}
  \item If $U_{\ell - k}$ is the coalescence of $v$ and $w$,
    i.e.\ if there exists $u$ such that
    $U_{\ell - k} = \{u\}$ and $\Gamma_{\!\mathrm{out}}(u) = \Set{v, w}$, then
    let $P_{k + 1} = \mathrm{coal}(P_k, \Set{\xi(v), \xi(w)})$.
  \item If $U_{\ell - k}$ is the reticulation of $u$ and $v$ with the hybrid $h$, 
    i.e.\ if $U_{\ell - k} = \Set{u', h', v'}$ with 
  $\Gamma_{\!\mathrm{out}}(u') = \{u, h'\}$, 
  $\Gamma_{\!\mathrm{out}}(v') = \{v, h'\}$ and
  $\Gamma_{\!\mathrm{out}}(h') = \{h\}$,
    then let ${P_{k + 1} = \mathrm{ret}(P_k,\, \xi(h),\, \Set{\xi(u), \xi(v)})}$.
\end{itemize}
The result of procedure is illustrated in Figure~\ref{figRTCN02}.
Note that the RTCN that produced a process
$(P_k, \, 1 \leq k \leq \ell)$ can unambiguously be recovered
from that process.

\begin{figure}[h!]
  \centering
  \captionsetup{width=0.95\linewidth}
  \includegraphics[width=0.8\linewidth]{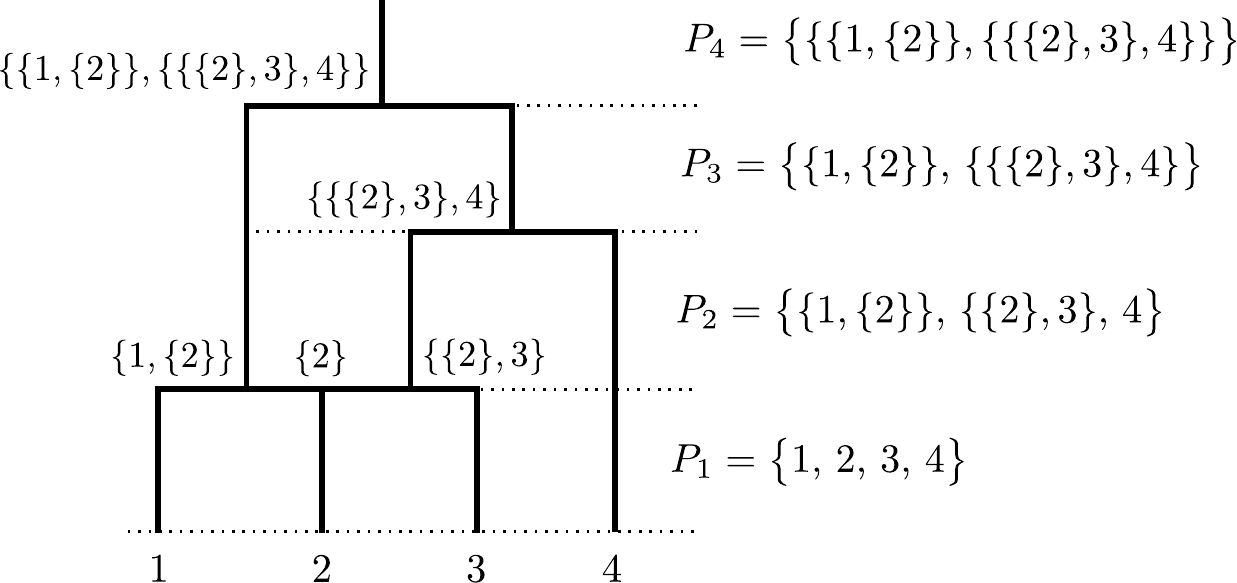}
  \caption{A RTCN with its canonical labeling and the associated process
  $(P_k, \, 1 \leq k \leq \ell)$.}  
\label{figRTCN02}
\end{figure}

In order to count RTCNs based on their number of reticulations
(or, equivalently, branchings, since both are linked by $r  + b = \ell - 1$),
we need to introduce the notion of profile of a RTCN.

\begin{definition} \label{defRTCNProfile}
Let $\nu$ be a RTCN and let $U_1 \prec \cdots \prec U_{\ell - 1}$ denote
its events. The \emph{profile} of~$\nu$
is the vector $\mathbf{q} = (q_1, \ldots, q_{\ell - 1})$ defined by
\[
  q_k =
  \begin{dcases}
    1, \text{ if $U_k$ is a branching;} \\
    0, \text{ otherwise} \, .
  \end{dcases}
\]
\revision{Note that we always have $q_1 = 1$.}
\end{definition}

We are now in position to prove our main result concerning the enumeration of
RTCNs.

\begin{theorem} \label{thmEnumRTCN}
The number of RTCNs with profile $\mathbf{q}$ is
\[
  F(\mathbf{q}) \,=\;
  \prod_{k = 1}^{\ell - 1}\big(q_k + (k - 1)(1 -q_k)\big)\, T_\ell  \, ,
\]
where $T_\ell = \ell!\,(\ell - 1)!\, / 2^{\ell - 1}$.
As a result, the number of RTCNs with $\ell$ labeled leaves and $b$
branchings is
\[
  C_{\ell, b} \;=\;
  {\ell - 1 \brack b}\, T_\ell\,,
\]
where the bracket denotes the unsigned Stirling numbers of the first kind.
\end{theorem}

\begin{proof}
Let us count the
processes $(P_k,\, 1 \leq k \leq \ell)$ with profile $\mathbf{q}$.
When going from $k + 1$ to $k$ lineages, i.e.\ when considering
event $U_k$, there are:
\begin{itemize}
  \item $\binom{k + 1}{2}$ possible coalescence events;
  \item $(k - 1) \binom{k + 1}{2}$ possible reticulation events.
\end{itemize}
Therefore,
\[
  F(\mathbf{q})
  \;=\; \prod_{k = 1}^{\ell - 1}\Big( \tbinom{k + 1}{2} q_k +
                               (k - 1) \tbinom{k + 1}{2} (1 - q_k)\Big)\, .
\]
Factoring out the binomial coefficients, we get the first part of the
theorem. Then, summing over all profiles with $b$ branchings,
\[
  C_{\ell, b}  \;=\;
  \mleft(\sum_{|\mathbf{q}| = b} \prod_{k = 1}^{\ell - 1}
  \big(q_k + (k - 1)(1 - q_k)\big)\mright)
  \times T_\ell \, .
\]
The first factor can be seen to be
the coefficient of degree~$b$ of the polynomial
$P(X) = (X + 0)(X + 1) \cdots (X + \ell - 2)$. Since by definition
of the unsigned Stirling
numbers of the first kind, 
\[
  P(X) = \sum_{b = 0}^{\ell - 1} {\ell - 1 \brack b} X^b \, ,
\]
this concludes the proof.
\end{proof}

Theorem~\ref{thmEnumRTCN} and its proof lead to several further results, which we now list.
%Let us start by pointing out the following procedure to sample uniform RTCNs
%conditional on a profile.

\begin{corollary} \label{corCountRTCN}
For $\ell \geq 2$, the number $C_\ell$ of RTCNs with $\ell$ labeled leaves is
\[
  C_\ell = \frac{\ell!\,{(\ell - 1) !\,}^2}{2^{\ell - 1}} \, .
\]
\end{corollary}
\begin{proof}[Proofs]
This follows from the expression for $C_{\ell, b}$ in Theorem \ref{thmEnumRTCN}
upon applying the identity
$\sum_{b = 0}^{\ell - 1}{\ell - 1 \brack b} = (\ell - 1)!$ 
\end{proof}

\begin{corollary}
The following procedure yields a uniform RTCN
with profile~$\mathbf{q}$.
Starting from $\ell$ labeled lineages, for $k = \ell - 1$ down to $1$:

\begin{tabular}{lll}
  & $\bullet$\; If $q_k = 1$:&
      let two lineages coalesce, uniformly at random \\
  & & among all $\binom{k + 1}{2}$ possibilities. \\[2ex]
  & $\bullet$\; If $q_k = 0$:&
      let three lineages reticulate, uniformly at random \\
  & & among all ${(k - 1)\binom{k + 1}{2}}$ possibilities.
\end{tabular}
\end{corollary}

%Second, let us lift the restriction on the
%number of reticulations.

\begin{proposition} \label{corUniformRTCNBackward}
Starting from $\ell$ labeled lineages, let pairs of lineages and triplets
of lineages reticulate, choosing what to do uniformly among all possibilities
at each step and stopping when there is only one lineage left. Then,
the resulting RTCN has the uniform distribution on the set of RTCNs with
$\ell$ labeled leaves.
\end{proposition}

\begin{proof}[Proofs]
Proposition~\ref{corUniformRTCNBackward} follows by noting that the realization of
the procedure correspond to those of the process
$(P_k, 1 \leq k \leq \ell)$, which uniquely encodes every leaf-labeled RTCN.
\end{proof}

Finally, let us point out the following fact about the profile and the
number of branchings of uniform RTCNs.

\begin{corollary} \label{propLawProfile}
Let $\mathbf{q}$ be the profile of a uniform RTCN with $\ell$ labeled
leaves. Then,
\[
  \mathbf{q} \;\sim\; (X_1, \ldots, X_{\ell - 1})\,, 
\]
where $(X_1, \ldots, X_{\ell - 1})$ are independent Bernoulli variables
such that
\[
  \Prob{X_k = 1} \;=\; \frac{1}{k}\, .
\]
As a result, the number $B_\ell$ of branchings of a uniform
RTCN with $\ell$ leaves, which is distributed as the number of cycles of
a uniform permutation of $\Set{1, \ldots, \ell - 1}$, satisfies
\begin{mathlist}
\item $\Expec{B_\ell} = H_{\ell - 1}$, where
  $H_{\ell - 1} = \sum_{k = 1}^{\ell - 1} 1/k$ is the $(\ell-1)$-th
  harmonic number.
\item As $\ell \to \infty$,
  $d_{\mathrm{TV}}\big(B_\ell,\, \mathrm{Poisson}(H_{\ell - 1})\big) \to 0$.
  In particular, $\frac{B_\ell - \log \ell}{\sqrt{\log \ell}}
  \tendsto[\,d\,]{} \mathcal{N}(0, 1)$.
\end{mathlist}
\end{corollary}

\begin{proof}
The first part of the proposition follows from the fact that the steps of the
algorithm described in Proposition~\ref{corUniformRTCNBackward} are
independent and that, when going from $k + 1$ to $k$ lineages there are
$(k - 1) \binom{k + 1}{2}$ possible reticulations and $\binom{k + 1}{2}$
possible coalescences, so that choosing uniformly among those the probability
of picking a coalescence is $1/k$.
  
Points (i) and (ii) for  $B_\ell$
are classic properties of the distribution of the number
of cycles in a uniform permutation -- see for instance Section~3.1 of
\cite{Pitman2006} -- that follow easily from its representation as a sum of
independent Bernoulli variables. Indeed, (i) is immediate and for (ii) we can
use the Stein-Chen bound on the total variation distance between a sum of
independent Bernoulli variables and the corresponding Poisson distribution
(recalled as Theorem~\ref{propRTCNPoissonApprox} in
Section~\ref{secRTCNRandPaths}) to get
\[
  d_{\mathrm{TV}}\big(B_\ell,\, \mathrm{Poisson}(H_{\ell - 1})\big) 
  \;\leq\;
  \min\Set{1, 1/H_{\ell - 1}} \sum_{k = 1}^{\ell - 1} \frac{1}{k^2}
  \;=\; O\mleft((\log\ell)^{-1}\mright)\,,  
\]
from which the central limit theorem follows readily
(see e.g.\ \cite{Barbour1992}, page~17).
\end{proof}

Finally, let us close this section by pointing out an unexpected connection
between RTCNs and a combinatorial structure known as ``river-crossings''.

\begin{remark} \label{remRiverCrossings}
The number of RTCNs with $\ell$ labeled leaves is also
the number of river-crossings using a two-person boat. It is recorded as
sequence \href{https://oeis.org/A167484}{A167484} in the Online Encyclopedia
of Integer Sequences~\cite{OEIS}, where it is described as follows:
\begin{quote}
  {\it For $\ell$ people on one side of a river, the number of ways they can all travel
 to the opposite side following the pattern of 2 sent, 1 returns, 2 sent, 1
 returns, ..., 2 sent.}
\end{quote}
\revision{We could not find a natural bijection between river-crossings and
RTCNs, and thought that there would not be one: indeed, for $\ell = 3$ the
$C_3 = 6$ river-crossings are completely equivalent up to permutation of the
labels, while the 6 RTCNs are not (3 of them contain a reticulation while 3
of them do not). However, while this paper was under revision
we were informed by Michael Fuchs (pers.\ comm.) that his group found one such
bijection, which they will describe in a forthcoming paper.}
\end{remark}

\subsection{Forward-time construction of RTCNs} \label{secRTCNOriented}

In this section, we give a forwards-in-time construction of RTCNs
that will yield a second proof of Theorem~\ref{thmEnumRTCN}.
This proof is more combinatorial than the one we have already given, and
will provide a different intuition as to why Stirling numbers arise in the
enumeration of RTCNs.

\revision{Here we introduce a further definition: Two vertices are said to be
\emph{siblings} if they have a parent in common,  and \emph{step-siblings} if
they have a sibling in common (thus, in a tree-child network two vertices are
step-siblings if and only if they are the tree-vertex children of the two
parents of a reticulation).}

The following notion will be useful.

\begin{definition}
A \emph{decorated RTCN} is a pair $(\nu, \theta)$, where
\begin{mathlist}
  \item $\nu$ is a RTCN with vertex set $V$ and root $\rho$.
  \item The \emph{decoration} is a function
    $\theta : V \setminus \{\rho\} \to \Set{0, 1}$
    such that
    \begin{itemize}
      \item If $v$ is a reticulation vertex or the child of a reticulation
        vertex, $\theta(v) = 0$.
      \item If $u$ and $v$ are siblings or step-siblings 
        and neither $u$ nor $v$ is a reticulation vertex, then
        $\theta(u) = 1 - \theta(v)$. \qedhere
    \end{itemize}
\end{mathlist}
\end{definition}

%Two vertices are said to be \emph{siblings} if they share a
%parent and \emph{step-siblings} if they share a sibling.

%ZZZ
Note that this formal definition is just a way to say that:
\begin{itemize}
  \item For every tree vertex $v$, we
    distinguish one of the outgoing edges by assigning a ``$1$'' to one
    of the children of $v$.
  \item For every reticulation vertex $v$, we distinguish one of the
    incoming edges by assigning a ``$1$'' to one of the siblings of $v$
    (this unambiguously determines their common parent $u$ and therefore
    the incoming edge $\vec{uv}$).
\end{itemize}

The notion of a decorated RTCN is similar in spirit to that of an
\emph{ordered ranked tree}, where an ordering is specified for the children of
each vertex -- and indeed both notions are equivalent for trees. However,
when working with RTCNs it is not useful to order the children of every
vertex, in part because reticulated vertices already play a special role.

\begin{figure}[h!]
  \centering
  \includegraphics[width=0.96\linewidth]{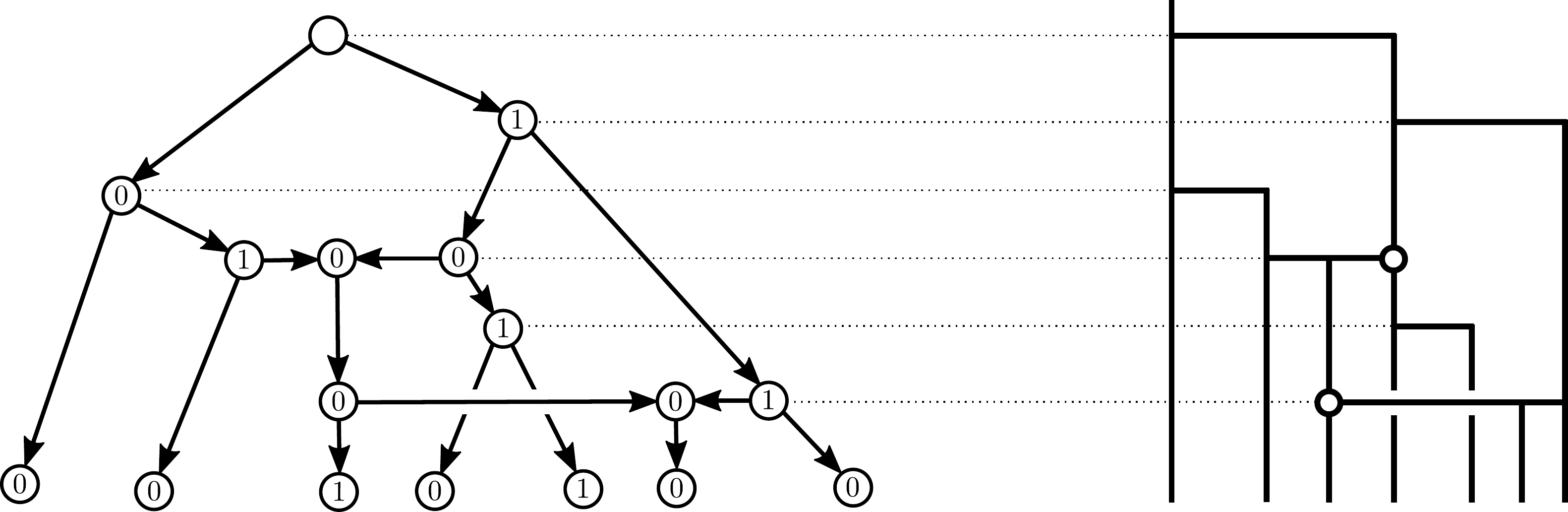}
  \captionsetup{width=0.96\linewidth}
  \caption{An example of a decorated RTCN (left) with its simplified representation
  (right). In the simplified representation, decorated lineages are indicated
  as follows: for branchings, they are drawn to the right; for reticulations,
  we add a white dot to indicate which parent of the reticulation points to a 1.}  
\label{figRTCN03}
\end{figure}

An important characteristic of decorated RTCNs is that, unlike their
undecorated counterparts, they are intrinsically labeled.
Indeed, it is possible to assign a unique label to every vertex of
a decorated RTCN $(\nu, \theta)$ as follows: first, remove every
reticulation edge distinguished by the decoration.
This yields a tree $\tilde{\nu}$. Then, assign to 
each vertex $v$ the label $\theta(u) \cdots \theta(v)$, where
$(\rho, u, \ldots, v)$ is the unique path going from the
root $\rho$ to $v$ in $\tilde{\nu}$. A consequence of this
unique labeling of internal vertices
is that there are $\ell!$ ways to label the leaves of a decorated
RTCN with $\ell$ leaves.

Before describing the forward-time construction of decorated RTCNs,
let us introduce one last combinatorial object.

\begin{definition} \label{defSubexcedant}
A \emph{subexcedant sequence of length $n$} is an integer-valued
sequence $s = (s_1, \ldots, s_n)$ such that, for all $k$, 
$1 \leq s_k \leq k$.
For any two subexcedant sequences $s$ and $s'$, the \emph{number of encounters
of $s$ and $s'$} is defined as
\[
  \mathrm{enc}(s, s') \;=\; \#\Set{k\geq1 \suchthat s_k = s'_k} \qedhere
\]
\end{definition}

\begin{lemma} \label{lemmaEncounters}
For any subexcedant sequence $s$ of length~$n$, there are
${n \brack k}$ subexcedant sequences $s'$ of length~$n$ such that
$\mathrm{enc}(s, s') = k$.
\end{lemma}

%A combinatorial proof of this classic result is recalled in
%Section~\ref{secRTCNAppPermut} of the Appendix.
%However, it also follows immediately by

This classic result follows immediately by considering a uniform subexcedant
sequence $s'$ and noting that its number of encounters with $s$ is distributed
as the sum for $k = 1$ to $n$ of independent Bernoulli variables with
parameters $1/k$.

Let us now describe how to encode decorated RTCNs using subexcedant sequences.
Let $s^\circ$ and $s^\bullet$ be two subexcedant sequences of length
$\ell - 1$. Start from a single lineage indexed~``1'' and, for $k = 1$ to
$\ell - 1$:
\begin{itemize}
  \item If $s^\circ_k = s^\bullet_k$, then let lineage~$s^\circ_k$ branch
    to create lineage $k + 1$.
  \item If $s^\circ_k \neq s^\bullet_k$, then let lineages
    $s^\circ_k$ and $s^\bullet_k$ hybridize to form lineage $k + 1$.
\end{itemize}
At each step of this procedure, decorate the lineage $s_k^\circ$
with a white dot. This construction is illustrated in Figure~\ref{figRTCN04}.

\begin{figure}[h!]
  \centering
  \includegraphics[width=0.9\linewidth]{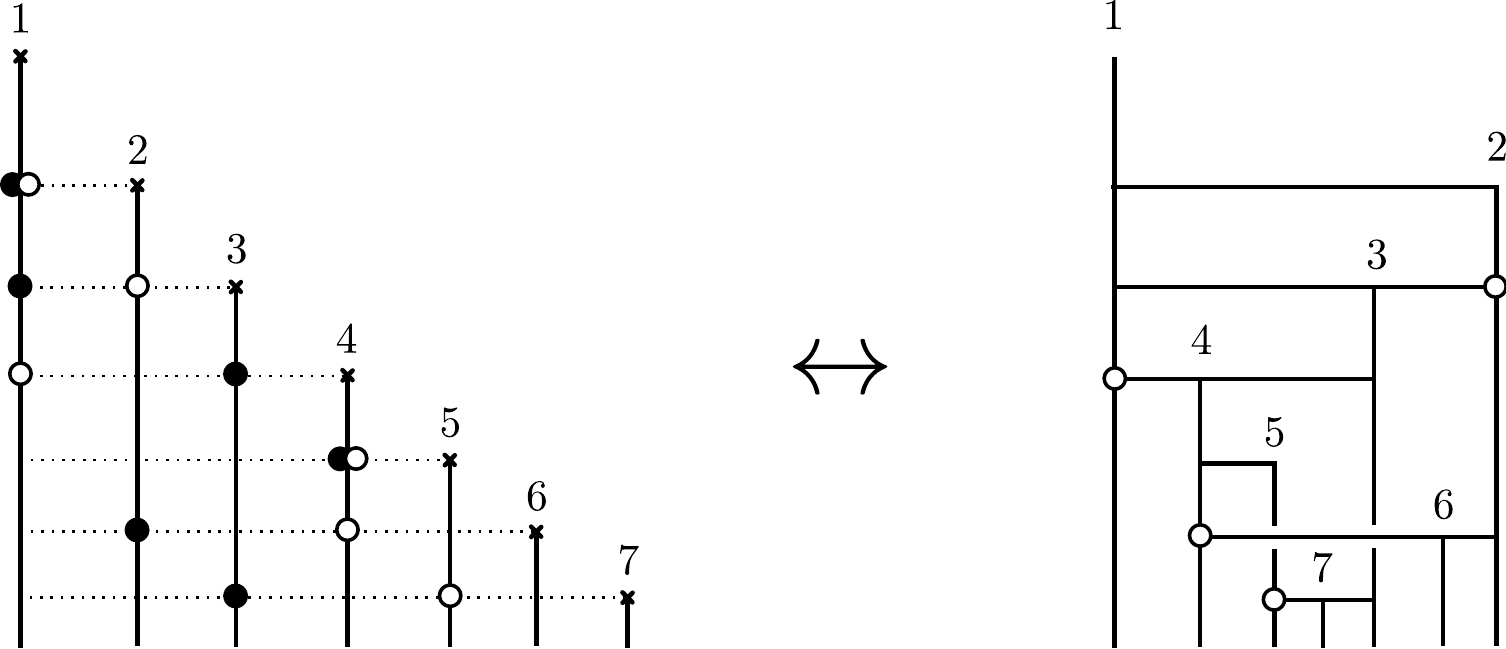}
  \captionsetup{width=0.9\linewidth}
  \caption{Graphical representation of the forward-time construction of
  RTCNs. On the left, the black and white dots represent the subexcedant
  sequences $s^\bullet$ and $s^\circ$, respectively, and on the right the
  corresponding RTCN. At each step, the new
  lineage (marked by a cross) is linked to the black dot and the white dot.
  If these two dots fall on the same lineage, we get a branching.}  
\label{figRTCN04}
\end{figure}

Observe that in this construction every pair of subexcedant sequences
$(s^\bullet, s^\circ)$ will yield a different decorated RTCN $\nu^\circ$, and
that given a decorated RTCN $\nu^\circ$ it is possible to unequivocally
recover the subexcedant sequences $(s^\bullet, s^\circ)$ that produced it.
Thus, letting $\mathscr{S}_{\ell - 1}$ denote the set of subexcedant
of length $\ell - 1$ and $\mathscr{C}_\ell^\circ$ that of decorated RTCNs with
$\ell$ unlabeled leaves, this constructions gives a
bijection between $\mathscr{S}_{\ell - 1} \times \mathscr{S}_{\ell - 1}$ and 
$\mathscr{C}_\ell^\circ$.

Also note that if at every step of the procedure
we only link lineage $k + 1$ to lineage $s_k^\circ$ and decorate the latter
with a white dot, we get a bijection between $\mathscr{S}_{\ell - 1}$
and the set $\mathscr{T}_{\ell}^\circ$ of decorated ranked trees with
$\ell$ unlabeled leaves.

From this forward-time encoding of decorated RTCNs
and Lemma~\ref{lemmaEncounters}, we immediately get the following proposition.

\begin{proposition} \label{propDecoratedRTCNs}
There are ${\ell - 1 \brack b}(\ell - 1)!$ decorated RTCNs with
$\ell$ unlabeled leaves and $b$ branchings.
\end{proposition}

To recover Theorem~\ref{thmEnumRTCN} from Proposition~\ref{propDecoratedRTCNs},
it suffices to recall that there are $\ell!$ ways to label the
leaves of a decorated RTCN and to note that there are $2^{\ell - 1}$ ways
to decorate a RTCN.

Finally, let us close this section by pointing out that the following
simple stochastic process generates uniform RTCNs.

\begin{proposition} \label{propForwardAlgoUnif}
Starting from a single lineage representing the root, 
let every lineage branch
at rate~1 and every ordered pair of lineages hybridize at rate 1,
decorating a lineage when it branches and decorating the first vertex
of an ordered pair of lineages when it hybridizes.
The RTCN obtained by stopping upon reaching $\ell$ lineages is uniform on
the set of decorated RTCN with $\ell$ leaves.
Relabeling its leaves uniformly at random and discarding its decoration
yields a uniform RTCN with $\ell$ labeled leaves.
\end{proposition}

%]]]

\section{RTCNs and ranked trees} \label{secRTCNLinkTrees}
%[[[

The forward-time construction of the previous section gave us
a way to encode a decorated RTCN using a decorated ranked tree and
a subexcedant sequence. In fact, Theorem~\ref{thmEnumRTCN} shows that it is
also possible to encode an \emph{undecorated} leaf-labeled RTCN
using a leaf-labeled ranked tree and, e.g., a subexcedant sequence or
a permutation -- even though we were unable to find any meaningful such
encoding.

In this section, we specialize the results of the previous section
in order to explain how to obtain RTCNs from ranked trees using simple graph
operations, without having to explicitly manipulate subexcedant sequences.  In
particular, we give a simple way to sample a uniform RTCN conditional on
displaying a given ranked tree. Readers who are more interested in the
structural properties or RTCNs than in how to sample them can jump to the
next section.

\subsection{Reticulation, unreticulation and base trees} \label{secRTCNModifs}

\enlargethispage{1ex}
Let us recall that the classic graph operation known as the
\emph{vertex identification} of $u$ and $v$ consists in replacing $u$ and $v$
by a new vertex $w$ whose neighbors (in-neighbors and out-neighbors, respectively)
are exactly the neighbors of $u$ and those of $v$ (without introducing a
self-loop in the case where $u$ and $v$ were neighbors).  Conversely, the
\emph{subdividing} of $\vec{uv}$ consists in introducing an intermediate vertex 
$w$ between $u$ and~$v$, that is: adding $w$ to the graph, removing $\vec{uv}$
and adding $\vec{uw}$ and~$\vec{wv}$.
\begin{definition}
The \emph{unreticulation of a reticulation edge $\vec{uv}$} is the graph
operation consisting in
\begin{enumerate}
  \item Removing $\vec{uv}$.
  \item Identifying $u$ and its (now only) child.
  \item Identifying $v$ and its child.
\end{enumerate}
Conversely, the \emph{reticulation of a tree edge $e_1$ into a tree edge $e_2$}
consists in
\begin{enumerate}
  \item Subdividing $e_1$ and $e_2$.
  \item Adding an edge from the vertex introduced in the middle of
    $e_1$ to that introduced in the middle of $e_2$.
\end{enumerate}
These operations are illustrated in Figure~\ref{figRTCN05}.
\end{definition}

\begin{figure}[h!]
  \centering
  \captionsetup{width=0.85\linewidth}
  \includegraphics[width=0.85\linewidth]{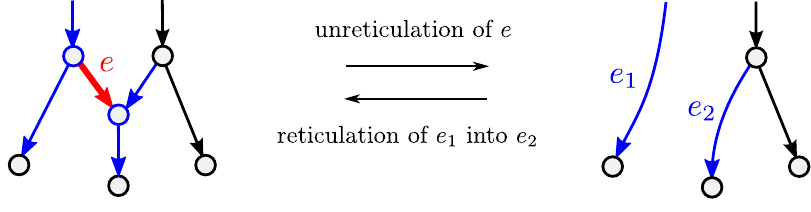}
  \caption{The operations of unreticulation and reticulation.}  
\label{figRTCN05}
\end{figure}

Unreticulating edges of a tree-child network $N$ yields a tree
$N'$ with the same number of leaves as $N$ as well as
a partition into events and a genealogical order that are compatible with
those of~$N$. This justifies the following definition.

\begin{definition} \label{defDisplay}
A ranked tree $\tau$ is a called a
\emph{base tree} of a RTCN $\nu$
if is possible to obtain it by unreticulating edges of $\nu$.
In that case, we write $\tau \sqsubset \nu$
and say that $\nu$ \emph{displays} $\tau$.
\end{definition}

Note that since each reticulation vertex has two incoming reticulating
edges and that unreticulating each of these yields different RTCNs,
every leaf-labeled RTCN with $r$ reticulations displays exactly $2^r$ ranked
trees. This is also a consequence of the fact that an RTCN is a binary normal network, and each such network displays exactly $2^r$ trees
(Corollary 3.4 of \cite{wil12}). 

\pagebreak

\subsection{Uniform base trees of a uniform RTCN}
\label{secRTCNToTrees}

\begin{proposition} \label{propUniformTreeFromRTCN}
The ranked tree obtained by unreticulating one the of two incoming edges of each
reticulation vertex of a uniform RTCN with $\ell$ labeled leaves,
each with probability~$1/2$, 
has the uniform distribution on the set of ranked trees with $\ell$ labeled
leaves.
\end{proposition}

\begin{proof} 
Let $\nu$ be a uniform leaf-labeled RTCN. The procedure
described in the proposition amounts to:
\begin{enumerate}
  \item Decorating $\nu$ uniformly at random to obtain
    a decorated leaf-labeled RTCN~$\nu^\circ$.
  \item Unreticulating each of the undecorated reticulation edges
    of $\nu^\circ$ to produce a decorated leaf-labeled tree $\tau^\circ$.
  \item Discarding the decoration of $\tau^\circ$.
\end{enumerate}
In the forward-time encoding, $\nu^\circ$ corresponds to a unique triplet
$\nu^\circ \simeq (s^\bullet, s^\circ, \sigma)$, where the permutation
$\sigma$ represents the leaf labeling. Now, since
$\nu^\circ$ is uniform on the set of decorated leaf-labeled RTCNs,
$(s^\circ, \sigma)$ is also uniform and as a result so is
$\tau^\circ \simeq (s^\circ, \sigma)$. Finally, the tree obtained by forgetting
the decoration of a uniform decorated leaf-labeled ranked tree
is uniform on the set of leaf-labeled ranked trees, concluding the proof.
\end{proof}

\subsection{Sampling RTCNs conditional on displaying a tree}
\label{secRTCNToTrees}

To obtain a RTCN from a ranked tree, we need to pay attention to the
constraints imposed by the chronological order. Indeed, it is not possible to
reticulate any tree edge of a RTCN into any other tree edge and obtain a RTCN.
To formulate this restriction, we need to introduce the notion of
contemporary edges.

\begin{definition} \label{defAliveEdges}
We now introduce a relation on the set $V$ of vertices of an RTCN.
Let us write $\bar{u} \preccurlyeq \bar{v}$ to indicate that
$\bar{u} \prec \bar{v}$ or $\bar{u} = \bar{v}$. In addition, 
if $u$ is a leaf then set $\bar{u} = \partial V$ and
$\bar{v} \prec \bar{u}$ for any internal vertex $v$.
  
The edge $\vec{uv}$ is said to be \emph{alive between two events $U \prec U'$}
if it meets the two following conditions:
\begin{mathlist}
  \item It is a tree edge.
  \item $\bar{u} \preccurlyeq U$ and $U' \preccurlyeq \bar{v}$.
\end{mathlist}
  
Two edges are said to be \emph{contemporary} if there exist two events
such that they are both alive between these events.
\end{definition}

As for events, these definitions become very intuitive when using the
graphical representation of RTCNs. Recall that, in this representation, we think
of the vertical axis as time, and that tree edges correspond to vertical lines
while events correspond to horizontal ones.  An edge is alive between two
events if a portion of it is located in the horizontal strip of the
plane delimited by the two events.  Two edges are contemporary if they
overlap when projected on the vertical axis. This is illustrated in
Figure~\ref{figRTNC06}.

\begin{figure}[h!]
  \centering
  \captionsetup{width=0.7\linewidth}
  \includegraphics[width=0.45\linewidth]{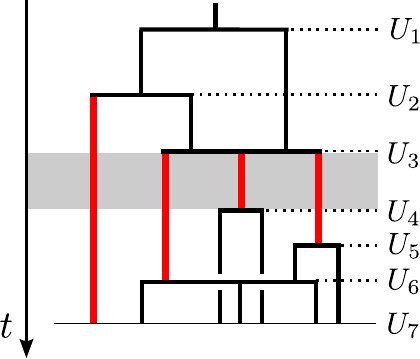}
  \caption{
  A RTCN with its events numbered in increasing order (and with
  the convention that $U_7 = \partial V$). The red edges are the four
  edges that are alive between events $U_3$ and $U_4$.}  
\label{figRTNC06}
\end{figure}

Note that if we let $U_1 \prec \ldots \prec U_{\ell - 1}$ denote the events of
a RTCN with $\ell$ leaves, with the convention that $U_{\ell} = \partial V$,
then there are exactly $k + 1$ edges alive between $U_k$ and $U_{k + 1}$.

Recall the Definition~\ref{defRTCNProfile} of a profile.

\begin{proposition} \label{propRTCNFromTree}
Let $\tau$ be a ranked tree with $\ell$ labeled leaves and let
$\mathbf{q}$ be a profile.
Letting $U_1 \prec \ldots \prec U_{\ell - 1}$ denote the branchings of $\tau$, 
for $k = 1$ to $\ell - 1$:
\begin{itemize}
  \item If $q_k = 1$, do nothing.
  \item If $q_k = 0$, letting $u$ denote the vertex such
    that $U_k = \{u\}$:
    \begin{enumerate}
      \item Pick an edge $e$ uniformly at random among the two
        outgoing edges of~$u$.
      \item Pick an edge $e'$ uniformly at random among the $k - 1$
        edges that are contemporary with $e$ and do not contain $u$.
      \item Reticulate $e'$ into $e$.
    \end{enumerate}
\end{itemize}
Then, resulting RTCN $\nu$ has the uniform distribution on the set of RTCNs with
profile $\mathbf{q}$ displaying $\tau$.
Moreover, if $\tau$ is uniform then $\nu$ is uniform on the set of 
RTCNs with profile $\mathbf{q}$. If in addition $\mathbf{q}$ is distributed
as the profile of a uniform RTCN (see Corollary~\ref{propLawProfile}), 
then the resulting RTCN is uniform on the set of RTCNs with $\ell$ labeled
leaves.
\end{proposition}

\begin{figure}[h!]
  \centering
  \captionsetup{width=0.85\linewidth}
  \includegraphics[width=0.85\linewidth]{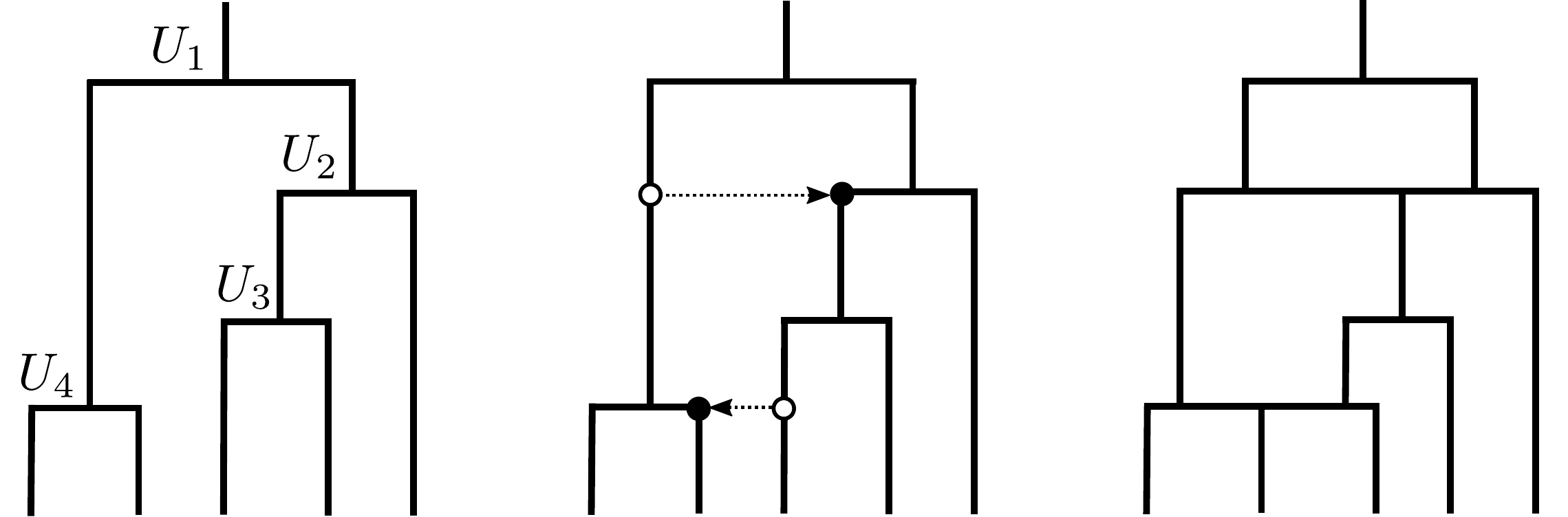}
  \caption{
  Example of the random construction of a RTCN from a ranked tree
  described in Proposition~\ref{propRTCNFromTree}. Here, nothing happens for
  $U_1$ and $U_3$. For $U_2$ and $U_4$, the black dot represents the edge
  $e$ and the white one the edge $e'$. Note that the modifications
  of branching events can be performed in any order, but that they have to be
  performed sequentially, so that contemporary edges remain well-defined at
  every step of the procedure.}  
\label{figRTCN07}
\end{figure}

\begin{proof}
The first part of the proposition follows from the fact (1) that for every
fixed ranked tree $\theta$ and for every RTCN $\mu$ displaying $\theta$
there is exactly one sequence of modifications of $\theta$ that yields $\mu$,
and (2) that each possible realization of the procedure described in the
proposition has the same probability, namely
\[
 \alpha(\mathbf{q}) = \prod_{k = 1}^{\ell - 1}
  \big(q_k + 2 (k - 1) (1 - q_k)\big)^{-1} \,, 
\]
although the exact value of this probability does not matter here.
Thus, letting $Q(\mu)$ denote the profile of $\mu$,
\[
  \Prob{\nu = \mu \given \tau = \theta} \;=\;
  \alpha(\mathbf{q})\,\Indic{Q(\mu) = \mathbf{q},\, \theta\, \sqsubset \mu}\,.
\]

Now let us assume that $\tau$ is uniform. We have to show that for any
RTCN $\mu$,
\[
  \Prob{\nu = \mu} \;=\; \tfrac{1}{F(\mathbf{q})}\Indic{Q(\mu) = \mathbf{q}}\,, 
\]
where $F(\mathbf{q})$ is the number of RTCNs with profile~$\mathbf{q}$, whose
exact value does not matter here.
Since $\Prob{\tau = \theta} = 1/T_\ell$ for all fixed
ranked tree $\theta$,
\[
  \Prob{\nu = \mu} \;=\;
  \sum_{\theta \sqsubset \mu}\;
  \alpha(\mathbf{q})
  \cdot
  \frac{1}{T_\ell}
  \cdot
  \Indic{Q(\mu) = \mathbf{q}} \, .
\]
Now let $r = \sum_{k = 1}^{\ell - 1} (1 - q_k)$ denote the number of
reticulations of the profile $\mathbf{q}$.
Since every RTCN with profile $\mathbf{q}$ displays exactly $2^r$ ranked trees,
we have
\[
  \Prob{\nu = \mu} \;=\;
  \frac{2^r \alpha(\mathbf{q})}{T_\ell}
  \cdot
  \Indic{Q(\mu) = \mathbf{q}} \,, 
\]
which concludes the proof since the factor
$2^r\alpha(\mathbf{q}) / T_\ell$ does not
depend on $\mu$.
\end{proof}

To close this section, note that to generate all RTCNs displaying a ranked tree
it suffices to apply the procedure above without the restrictions of the
profile, and that there are then $2(k - 1) + 1$ possible actions at step
$k$ of the procedure. This gives us the following proposition counting
the number of RTCNs displaying a ranked tree.

\begin{proposition}
\label{pro:double}
The number of RTCNs displaying any ranked tree $\tau$ is
\[
   \#\Set{\nu \suchthat \tau \sqsubset \nu} 
   \;=\;
   \prod_{k = 1}^{\ell - 1} (2k - 1)
   \;=\; (2\ell - 3)!!
\]
\end{proposition}

\begin{remark}
Surprisingly, the right-hand side of the expression in
Proposition~\ref{pro:double}  is also the number of rooted unranked binary
trees with $\ell$ labeled leaves --
see e.g.\ \cite{LambertBrazJProbabStat2017} or \cite{Steel2016}. While it
is easy to give a bijective proof of this, we have not found a ``graphical''
bijection that would make this intuitive.
\end{remark}

%]]]

\section{Cherries and tridents} \label{secRTCNStats}
%[[[

Cherries and tridents are among the most basic statistics
of tree-child networks. In biological terms,  a cherry is a pair of non-hybrid
sibling species and a trident is a group of three extant species
such that one of these species was produced by the hybridization of the two
others.  These notions can be formalized as follows.

\begin{definition}
An event is said to be \emph{external} when
\[
  \bigcup_{u \in U} \Gamma_{\!\mathrm{out}}(u) \subset \partial V\,,  
\]
where $\Gamma_{\!\mathrm{out}}(u)$ denotes the set of children of vertex~$u$
and $\partial V$ is the leaf set of~$N$.
\begin{itemize}
  \item A \emph{cherry} is an external branching event.
  \item A \emph{trident} is an external reticulation event.
\end{itemize}
These notions are illustrated in Figure~\ref{figExCherriesRetCherries}.
\end{definition}

\begin{figure}[h!]
  \centering
  \captionsetup{width=0.47\linewidth}
  \includegraphics[width=0.4\linewidth]{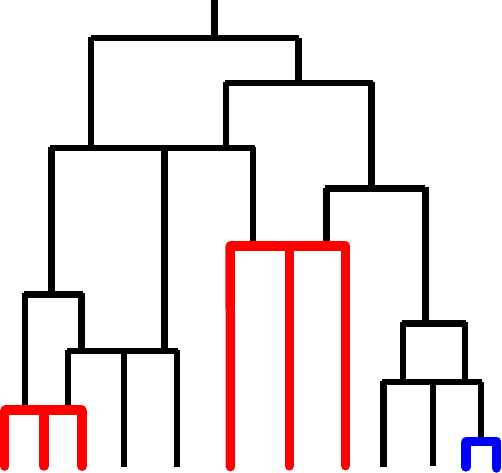}
  \caption{A RTCN with its only cherry (in blue) and its two tridents (in red).}  
\label{figExCherriesRetCherries}
\end{figure}
\begin{center}
\end{center}

\subsection{Number of cherries} \label{secRTCNCherries}

In this section, we prove the first part of
what was announced in the introduction as
Theorem~\ref{thmRTCNCherriesRetCherries}.

\begin{theorem}[Part (i)]
\label{thmRTCNCherriesRetCherries}
Let $\kappa_\ell$ be the number of cherries of a uniform RTCN with
$\ell$ labeled leaves. Then,
\[
  \kappa_\ell\; \tendsto[d]{\ell \to \infty}\;
  \mathrm{Poisson}\mleft(\tfrac{1}{4}\mright) \,.
\]
\end{theorem}

The proof consists in coupling $\kappa_\ell$ with a Markov chain to show
that the factorial moments of $\kappa_\ell$ (i.e.
$m_\ell^n \defas \Expec{\kappa_\ell(\kappa_\ell - 1)\cdots(\kappa_\ell - n + 1)}$) 
satisfy the recursion:
\[
  m_{\ell + 1}^n \;=\;
  \mleft(\frac{\ell - 2n}{\ell}\mright)^2 m_\ell^n \;+\;
  \frac{n (\ell - 2n + 2)}{\ell^2}\, m_\ell^{n - 1} \,,
\]
and then using this recursion to prove that $m^{n}_\ell \to 1/4^n$ as
$\ell \to \infty$ for all $n$. We separate each of these steps into different
propositions.

\begin{proposition} \label{propRTCNCherriesChain}
Let $(X_\ell)_{\ell \geq 2}$ be the Markov chain defined by:
\begin{mathlist}
\item $X_2 = 1$.
\item For $\ell \geq 2$, conditional on $X_\ell = k$,
  \[
    X_{\ell + 1} \;=\;
    \begin{dcases}
      k - 2,  \text{with probability} & \frac{2k(2k - 2)}{\ell^2};\\ 
      k - 1, \,\qquad/\!\!/ & \frac{2k + 4k(\ell - 2k)}{\ell^2}\\ 
      k,  \qquad\qquad/\!\!/ & \frac{(\ell - 2k)(\ell - 2k - 1) + 2k}{\ell^2};\\
      k + 1, \qquad\;/\!\!/&\frac{\ell - 2k}{\ell^2}.
    \end{dcases}
  \]
\end{mathlist}
Then, for all $\ell\geq 2$, $X_\ell$ is distributed as the number
$\kappa_\ell$ of cherries of a uniform RTCN with $\ell$ labeled leaves.
\end{proposition}

\begin{proof}
Let us use the forward-time construction to build a nested sequence
$(\nu^\circ_{\ell})_{\ell \geq 2}$ of uniform decorated RTCNs.
Recall from Section~\ref{secRTCNOriented} that to go from $\nu_{\ell}^\circ$ to
$\nu_{\ell + 1}^\circ$, we choose an ordered pair $(u, v)$ of lineages
$\nu_\ell^\circ$ uniformly at random. If
$u = v$, we let the lineage branch; otherwise we let $u$ and $v$ hybridize.

Now, assume that $\nu^\circ_\ell$ has $k$ cherries, so that
there are $2k$ lineages that belong to cherries (which we refer to
as C leaves) and $\ell - 2k$ lineages that do not (non-C leaves).
Let us list all possible choices of $(u, v)$ and see their effect on the
number of cherries.
\begin{enumerate}
\item If the next event is a branching:
\begin{mathlist}
\item If a C leaf is chosen, the number cherries does not change. This
  happens with probability $2k / \ell^2$.
\item If a non-C leaf is chosen, one cherry is created (probability:
  $(\ell - 2k)/\ell^2$).
\end{mathlist}
\item If the next event is a hybridization:
\begin{mathlist}
\item If the two leaves of a same cherry are chosen, that
  cherry is destroyed (probability: $2k / \ell^2$).
\item If two leaves of two different cherries are chosen, these two
  cherries are destroyed (probability: $2k (2k - 2) / \ell^2$).
\item If a C leaf and a non-C leaf are chosen,
  one cherry is destroyed (probability: $4k(\ell - 2k) / \ell^2$).
\item If two non-C leaves are chosen, the number of cherries does not change 
  (probability: $(\ell - 2k)(\ell - 2k - 1) / \ell^2$).
\end{mathlist}
\end{enumerate}

Doing the book-keeping and observing that a RTCN with two leaves has one
cherry concludes the proof.
\end{proof}

\vspace{-1.1ex}

\begin{notation}
We denote by $\fallfact{x}{n} = \prod_{k = 0}^{n - 1}(x - k)$ the
$n$-th falling factorial of $x$.
\end{notation}

\begin{proposition} \label{propRTCNCherriesRecur}
Let $(X_\ell)_{\ell \geq 2}$ be the Markov chain defined in
Proposition~\ref{propRTCNCherriesChain} and let
$m_\ell^n = \Expec{\fallfact{X_\ell}{n}}$ denote the $n$-th factorial moment
of $X_\ell$. Then,
\[
  m_{\ell + 1}^n \;=\;
  \mleft(\frac{\ell - 2n}{\ell}\mright)^2 m_\ell^n \;+\;
  \frac{n (\ell - 2n + 2)}{\ell^2}\, m_\ell^{n - 1} \,.
\]
\end{proposition}

\begin{proof}
Let $p_{(-2)}, p_{(-1)}, p_{(+0)}$ and $p_{(+1)}$ denote the transition
probabilities of $X_\ell$, conditional on $X_\ell = k$. Then,
\enlargethispage{3ex}
\begingroup
\allowdisplaybreaks % IF NEEDED (but try to avoid)
\begin{align*}
  &\hspace{-2ex} \Expec{\fallfact{X_{\ell + 1}}{n} \given X_\ell = k} \\[2ex]
  &=
  \fallfact{(k - 2)}{n}\, p_{(-2)} \;+\;
  \fallfact{(k - 1)}{n}\, p_{(-1)} \;+\;
  \fallfact{k}{n}\, p_{(+0)} \;+\;
  \fallfact{(k + 1)}{n}\, p_{(+1)} \\[2ex]
  &=\fallfact{k}{n}\; \mleft(
  \frac{(k - n)(k - n - 1)}{k(k - 1)} p_{(-2)}\;+\;
  \frac{k - n}{k} p_{(-1)}\;+\;
  p_{(+0)}\;+\;
  \frac{k + 1}{k - n + 1} p_{(+1)}\mright) \\[1ex]
  &=\frac{\fallfact{k}{n}}{\ell^2} \; \Bigg(
  4 (k - n)(k - n - 1)\;+\; 
  2(k - n)(1 + 2(\ell-2k))\\[-2ex]
  &\hspace{28ex} \;+\; (\ell - 2k)(\ell - 2k - 1) + 2k +
  \frac{(k + 1)(\ell - 2k)}{k - n + 1} \Bigg).
\end{align*}  
\endgroup
After a little algebra to rearrange this last expression, we get
\[
  \Expec{\fallfact{X_{\ell + 1}}{n} \given X_\ell = k} \;=\;
  \fallfact{k}{n}\,\mleft(\frac{\ell - 2n}{\ell}\mright)^2 \;+\;
  \fallfact{k}{n - 1}\;\frac{n (\ell + 2 - 2n)}{\ell^2},
\]
and the proposition follows by integrating in $k$.
\end{proof}

\begin{proposition} ~
\begin{mathlist}
\item For all $\ell \geq 2$,
  $\Expec{\kappa_\ell} = \frac{(3\ell - 5) \ell}{12 (\ell - 1)(\ell - 2)}$.
\item For all $n \geq 1$, $\Expec{\fallfact{\kappa_\ell}{n}} \to 1/4^n$ as
  $\ell \to \infty$. As a result,
  $\kappa_\ell \tendsto[d]{\ell \to \infty} \mathrm{Poisson}(1/4)$.
\end{mathlist}
\end{proposition}

\begin{proof}
The expression for $\Expec{\kappa_\ell}$ given in (i) follows from
Lemma~\ref{lemmaRTCNRecur} and routine calculations.
To prove (ii), we proceed by induction on $n$. First, we see from the
expression in (i) that $m_\ell^1 \to 1/4$ as $\ell \to \infty$.
Now suppose that $n\geq 2$ and that $m_\ell^{n - 1} \to 1/4^{n-1}$ as $\ell \to \infty$.
By Proposition~\ref{propRTCNCherriesRecur}, 
\[
  m_{\ell + 1}^n \;=\;
  a_\ell \, m_\ell^n \;+\; b_\ell\,, 
\]
where
\[
  a_\ell = \mleft(\frac{\ell - 2n}{\ell}\mright)^2
  \quad\text{and}\quad
  b_\ell = 
  \frac{n (\ell - 2n + 2)}{\ell^2} \, m_\ell^{n - 1} \, .
\]
Since $a_\ell \neq 0$ for all $\ell \geq 2n + 1$, 
application of Lemma~\ref{lemmaRTCNRecur} together with the identity
\[
  \prod_{j = 2n + 1}^{k}
  \mleft(\frac{j - 2n}{j}\mright)^2 \;=\;
  \frac{1}{\binom{k}{2n}^2} 
\]
gives
\[
  m_{\ell}^{n} \;=\;\mleft(m_{2n + 1}^n + \sum_{k = 2n + 1}^{\ell - 1}
  \frac{n (k - 2n + 2)}{k^2} \, \binom{k}{2n}^2 \, m_k^{n - 1} \, 
  \mright) \;\Big/\; \binom{\ell - 1}{2n}^2.
\]
Now, as $k\to \infty$, $\binom{k}{2n}\sim k^{2n} / (2n)!$
and, by the induction hypothesis, ${m_k^{n - 1} \sim 1/4^{n - 1}}$. As
a result,
\[
  \frac{n (k - 2n + 2)}{k^2} \, \binom{k}{2n}^2 \, m_k^{n - 1} \;\sim\;
  \frac{n /4^n}{(2n)!\,^2} \; k^{4n - 1} \, .
\]
Using Lemma~\ref{lemmaStupid} to get an asymptotic equivalent of the sum
of these terms,
\[
  m_{\ell}^{n} \;\sim\;
  \frac{n / 4^{n - 1}}{(2n)!\,^2 } \cdot \frac{\ell^{4n}}{4n}\cdot
    \mleft(\frac{(2n)!}{\ell^{2n}}\mright)^2 \;=\;
   1/4^n \, .
\]
The convergence in distribution of $X_\ell$ to a Poisson distribution
with mean $1/4$ is then a classic result (see e.g.\ Theorem~2.4
in \cite{VanDerHofstad2016}).
\end{proof}

\subsection{Number of tridents} \label{secRTCNsRetCherries}

In this section, we prove the second part of
Theorem~\ref{thmRTCNCherriesRetCherries}.

\begin{reptheorem}{thmRTCNCherriesRetCherries}[Part (ii)]
Let $\chi_\ell$ be the number of tridents of a uniform RTCN with
$\ell$ labeled leaves. Then,
\[
  \frac{\chi_\ell}{\ell}\; \tendsto[\P]{\ell \to \infty}\; \frac{1}{7} \, .
\]
\end{reptheorem}

The proof is quite similar to that used in the previous section to study
the number $\kappa_\ell$ of cherries --
namely, we couple $\chi_\ell$ with a Markov chain in order to compute its
moments. The difference is that the moments of $\chi_\ell$ are not as tractable
as those of $\kappa_\ell$. As a result, we only compute the first two
moments explicitly and then use Chebyshev's inequality to prove the
convergence in probability.

\begin{proposition} \label{propRTCNMarkovchainRetCherries}
Let $(X_\ell)_{\ell \geq 2}$ be the Markov chain defined by
\begin{mathlist}
\item $X_2 = 0$
\item For $\ell \geq 2$, conditional on $X_\ell = k$,
  \[
    X_{\ell + 1} \;=\;
    \begin{dcases}
      k - 1, \text{ with probability } & \frac{3k\, (3k - 2)}{\ell^2};\\ 
      k,  \qquad\qquad/\!\!/ & 1 -
      \frac{3k\, (3k - 2) + (\ell - 3k)(\ell - 3k - 1)}{\ell^2};\\
      k + 1, \qquad\;/\!\!/&\frac{(\ell - 3k)(\ell - 3k - 1)}{\ell^2}.
    \end{dcases}
  \]
\end{mathlist}
Then, for all $\ell\geq 1$, $X_\ell$ has the distribution of the number of
tridents of a uniform RTCN with $\ell$ leaves.
\end{proposition}

\begin{proof}
As in the proof of Proposition~\ref{propRTCNCherriesChain}, let us
consider a nested sequence of uniform decorated
RTCNs $(\nu_\ell^\circ)_{\ell \geq 2}$
produced by the forward-time construction, and see how the number of
tridents is affected when we go from
$\nu_\ell^\circ$ to $\nu_{\ell + 1}^\circ$.

Assuming that there are $k$ tridents in $\chi_\ell$, there are
$3k$ leaves associated with tridents (RC leaves) and
$\ell - 3k$ other leaves (non-RC leaves). The possible cases are the following:
\pagebreak
\begin{enumerate}
  \item The next event is a branching:
    \begin{mathlist}
      \item If a non-RC leaf is chosen, the number of tridents
        does not change. This happens with probability $(\ell - 3k) / \ell^2$.
      \item If a RC leaf is chosen, the corresponding trident
        is destroyed (probability: $3k / \ell^2$).
    \end{mathlist}
  \item The next event is a reticulation:
    \begin{mathlist}
       \item If two non-RC leaves are chosen, one trident is created
         (probability: $(\ell - 3k)(\ell - 3k - 1)/\ell^2$).
       \item If a non-RC leaf and a RC leaf are chosen, one trident
         is destroyed and one is created
         (probability: $6k(\ell - 3k)$).
       \item If two RC leaves are chosen: 
         \begin{enumerate}[a.]
           \item If they do not belong to two different tridents,
             these are destroyed and a new
             trident is created
             (probability: $3k(3k - 3) / \ell^2$).
            \item If they belong to the same trident, 
              this trident is destroyed and another one is created
              (probability: $6k/\ell^2$).
         \end{enumerate}
    \end{mathlist}
\end{enumerate}
Doing the book-keeping and observing that a RTCN with two leaves has
zero tridents yields the proposition.
\end{proof}

\begin{proposition} \label{propRTCNExpecRetCherries}
The expected number $\mu_\ell = \Expec{\chi_\ell}$ of tridents
of a uniform RTCN with $\ell$ leaves 
$\mu_\ell$ satisfies the recursion 
\[
  \mu_{\ell + 1} = \mleft(\frac{\ell - 3}{\ell}\mright)^2 \mu_\ell \;+\;
  \frac{\ell - 1}{\ell} \,.
\]
As a result, we have $\mu_2 = 0$, $\mu_3 = 1/2$, $\mu_4 = 2/3$ and, for
$\ell \geq 4$,
\[
  \mu_\ell \;=\; \frac{(15 \, \ell^3 - 85 \, \ell^2 + 144 \, \ell - 71)\, \ell}
                      {105 \, (\ell - 1) (\ell - 2) (\ell - 3)} \,.
\]
\end{proposition}

\begin{proof}
Using the Markov chain of Proposition~\ref{propRTCNMarkovchainRetCherries},
we have
\begin{align*}
 \Expec{X_{\ell + 1} \given X_\ell = k}
  =& \;\; k \; +\;  \Prob{X_{\ell + 1} = k + 1 \given X_\ell = k} -
  \Prob{X_{\ell + 1} = k - 1 \given X_\ell = k} \\
  =&\;\;k\,\mleft(\frac{\ell - 3}{\ell}\mright)^2 \;+\;
  \frac{\ell - 1}{\ell} \,,
\end{align*}
and the recursion follows by integrating against the law of $X_\ell$.

The expression of $\mu_\ell$ then follows from
Lemma~\ref{lemmaRTCNRecur} and calculations that are better
performed by a symbolic computation software such as \cite{Sage}.
\end{proof}

\begin{proposition} \label{propRTCNVariance}
The variance of the number of tridents of a uniform RTCN
with $\ell$ leaves is
\[
  \Var{\chi_\ell} \;=\; \frac{24}{637}\, \ell \;+\; \frac{1}{21} \;+\; o(1) \,.
\]
\end{proposition}

The proof of this proposition is exactly the same as that of
Proposition~\ref{propRTCNExpecRetCherries} but involves more complex expressions
that can be found in Section~\ref{appRTCNVariance} of the Appendix.

Finally, the convergence in probability of $\chi_\ell / \ell$ to $1/7$
follows readily from Chebyshev's inequality and
the fact that $\Expec{\chi_\ell} \sim \ell/7$ and
$\Var{\chi_\ell} = O(\ell)$.

%]]]

\section{Random paths between the root and the leaves} \label{secRTCNRandPaths}
%[[[

In this section, we study the length of two random paths going from the root
to the leaf set:
\begin{enumerate}
  \enlargethispage{1ex}
  \item A path obtained by starting from the root and going ``down'' towards
    the leaves, choosing each outgoing edge with equal probability whenever
    we reach a tree vertex.
  \item A path obtained by starting from a uniformly chosen
    leaf and going ``up'' towards
    the root, choosing each incoming edge with equal probability whenever we
    reach a reticulation vertex.
\end{enumerate}

\begin{definition} \label{defLengthPath}
The \emph{length} of a path $\gamma$ is its number of tree edges.
\end{definition}

The reason why we do not count reticulation edges
when calculating the length of a path is that,
from a biological point of view,
reticulation edges are supposed to correspond to ``instantaneous''
hybridization events. \revision{As will become apparent from their proof,
our results can be adapted to the case where all edges are
counted. However, this yields a
compound Poisson distribution as the approximation distribution fork
the lengths of the paths (see e.g.\ Lemma~8 in~\cite{boutsikas2000bound}).}

Before starting with the proofs, let us introduce some notation.

\begin{notation}
We denote by
\[
  H^{(m)}_n = \sum_{k = 1}^n \frac{1}{k^m}
\]
the $n$-th generalized harmonic number of order $m$. We also use the notation
{$\eulergamma = \lim_n H_n^{(1)} - \log n$} for the Euler-Mascheroni constant,
\revision{where as in the rest this text ``$\log$'' denotes the natural
logarithm.}
\end{notation}
  
Finally, let us recall a classic bound on the total variation distance
between a sum of independent Bernoulli variables and the
corresponding Poisson distribution.

\begin{quotedtheorem} \label{propRTCNPoissonApprox}
Let $X_1, \ldots, X_n$ be independent Bernoulli variables with parameters
$\Prob{X_i = 1} = p_i$, and let $\lambda_n = \sum_{i = 1}^{n} p_i$. Then,
\[
  d_{\mathrm{TV}}\mleft(\sum_{i = 1}^n X_i,\; \mathrm{Poisson}(\lambda_n)\mright)
  \;\leq\;
  \min(1, 1/\lambda_n) \sum_{i = 1}^n p_i^2 \,, 
\]
where $d_{\mathrm{TV}}$ denotes the total variation distance.
\end{quotedtheorem}

This inequality is a consequence of the Stein-Chen method and can be found,
for example, as Theorem~4.6 in \cite{RossProbabilitySurveys2011}.

\subsection{Length of a random walk from the root to a leaf} \label{secRTCNLenghtDown}

In this section, we prove the first part of 
what was announced as Theorem~\ref{thmRTCNRandPaths} in the introduction.

\begin{theorem}[Point (i)]  \label{thmRTCNRandPaths}
Let $\nu$ be a uniform RTCN with $\ell$ leaves, and
let $\gamma^{\downarrow}$ be a random path obtained by starting from the
root and following the edges of $\nu$, choosing
each of the two out-going edges of a tree vertex with equal probability and
stopping when we reach a leaf. Then,
\[
  \mathrm{length}(\gamma^{\downarrow}) \;=\;
  \sum_{k = 1}^{\ell - 1} I_k \,, 
\]
where $I_1, \ldots, I_{\ell - 1}$ are independent Bernoulli variables with
parameter 
\[
  \Prob{I_k = 1} = \frac{2k - 1}{k^2}\,.
\]
In particular, letting $c^{_\downarrow} = 2 \eulergamma - \pi^2/6$,
where $\eulergamma$ is the Euler-Mascheroni constant,
\begin{mathlist}
\item $\Expec{\mathrm{length}(\gamma^{\downarrow})} =
  2\log \ell + c^{_\downarrow} + o(1)$.
\item $\Var{\mathrm{length}(\gamma^{\downarrow})} =
  2\log \ell + O(1)$.
\item $d_{\mathrm{TV}}\big(\mathrm{length}(\gamma^{\downarrow}),\,
  \mathrm{Poisson}(2\log \ell + c^{_\downarrow})\big) \to 0$ as $\ell \rightarrow \infty$.
\end{mathlist}
\end{theorem}

\begin{proof}
The idea of the proof is to use the forward-time construction to build
jointly a nested sequence $(\nu_\ell^\circ)_{\ell \geq 2}$ of uniform
decorated RTCNs and the random path $\gamma^\downarrow$.
With the convention that $\nu_1^\circ$ consists of a single lineage,
for $k \geq 2$ let $(u_k, v_k)$ denote the pair of lineages that was
chosen to turn $\nu^\circ_{k - 1}$
into $\nu^\circ_{k}$ (recall that if $u_k = v_k$ then the next event is a
branching) and let $x_k$ record the position of the random walk among the
leaves of $\nu^\circ_{k - 1}$. With this notation, the length of
$\gamma^\downarrow$ in $\nu^\circ_\ell$ is
\[
  \mathrm{length}(\gamma^{\downarrow}) \;=\;
  \sum_{k = 1}^{\ell - 1} \Indic{x_{k} \in \{u_k, v_k\}}\,, 
\]
where the variables $\Indic{x_{k} \in \{u_k, v_k\}}$ are independent
because $(x_{k-1}, x_k)$ is independent of $(u_k, v_k)$.
Moreover, since $(u_k, v_k)$ is chosen uniformly among the pairs of lineages
of $\nu^\circ_{k - 1}$ and independently of $x_k$,
\[
  \Prob{x_{k} \in \{u_k, v_k\}} = \frac{2k - 1}{k^2}\,, 
\]
which proves the first part of the proposition.

The rest of the proposition follows immediately from
Theorem~\ref{propRTCNPoissonApprox} since, letting $p_k = (2k - 1)/k^2$,
\begin{itemize}
  \item $\Expec{\mathrm{length}(\gamma^\downarrow)} =
    \sum_{k = 1}^{\ell - 1} p_k =
    2 H^{(1)}_{\ell - 1} - H^{(2)}_{\ell - 1}$.
  \item $\Var{\mathrm{length}(\gamma^\downarrow)} =
    \sum_{k = 1}^{\ell - 1} p_k (1 - p_k) =
    2 H^{(1)}_{\ell - 1} - 5 H^{(2)}_{\ell - 1} + 4
    H^{(3)}_{\ell - 1} - H^{(4)}_{\ell - 1}$.
  \item $\sum_{k = 1}^{\ell - 1} p_k^2 =
    4 H^{(2)}_{\ell - 1} - 4 H^{(3)}_{\ell - 1} + H^{(4)}_{\ell - 1} = O(1)$ .\qedhere
\end{itemize}
\end{proof}

\subsection{Length of a random walk from a leaf to the root} \label{secRTCNLengthUp}

In this section, we prove the second part of Theorem~\ref{thmRTCNRandPaths}.

\begin{reptheorem}{thmRTCNRandPaths}[Point (ii)]
Let $\nu$ be a uniform RTCN with $\ell$ leaves, and
let $\gamma^{\uparrow}$ be a random path obtained by starting from a uniformly
chosen leaf and following the edges of $\nu$ in reverse direction, choosing
each of the two incoming edges of a reticulation vertex with equal probability
and stopping when we reach the root. Then,
\[
  \mathrm{length}(\gamma^{\uparrow}) \;=\;
  \sum_{k = 2}^{\ell} J_k \,, 
\]
where $J_2, \ldots, J_{\ell}$ are independent Bernoulli variables with
parameter 
\[
  \Prob{J_k = 1} = \frac{3k - 4}{k(k - 1)}\,.
\]
In particular, letting $c^{_\uparrow} = 3 \eulergamma - 4$, where
$\eulergamma$ is the Euler-Mascheroni constant,
\begin{mathlist} \label{propRTCNUpPath}
\item $\Expec{\mathrm{length}(\gamma^{\uparrow})} = 3 \log \ell
  + c^{_\uparrow} + o(1)$.
\item $\Var{\mathrm{length}(\gamma^{\uparrow})} = 3\log \ell + O(1)$.
\item $d_{\mathrm{TV}}\big(\mathrm{length}(\gamma^{\uparrow}),\,
  \mathrm{Poisson}(3\log \ell + c^{_\uparrow})\big) \to 0$.
\end{mathlist}
\end{reptheorem}

\begin{remark}
Note that the random path $\gamma^\uparrow$ is not uniformly chosen
among all the paths going from the focal leaf to the root.
\end{remark}

\begin{proof}
The proof is similar to that of the previous section, but
this time the idea is to use the backward-time construction to jointly
build the RTCN $\nu$ and the random path $\gamma^\uparrow$.
Recall that, in the backward-time construction, for $k = \ell$ down to $2$,
we go from $k$ to $k - 1$ lineages by choosing an event uniformly
at random among the $k(k - 1)/2$ possible coalescences and $k(k - 1)(k - 2)/2$
possible reticulations. Out of these, $k - 1$ coalescences and
$3(k - 1)(k - 2)/2$ reticulations involve the lineage through which
$\gamma^\uparrow$ goes, and the choice is independent of the position of
$\gamma^\uparrow$. As a result, the probability that the lineage
containing~$\gamma^\uparrow$ is involved in the event that is chosen is
\[
  \frac{k - 1 + 3(k - 1)(k - 2) / 2}{k(k - 1) / 2 + k(k - 1)(k - 2)/2}
  \;=\;
  \frac{3k - 4}{k(k - 1)} \,, 
\]
proving the first part of the proposition.  The rest of the proposition then
follows from Theorem~\ref{propRTCNPoissonApprox} and from the fact that,
letting $p_k = \frac{3k - 4}{k(k - 1)}$,
\begin{itemize}
  \item $\Expec{\mathrm{length}(\gamma^{\uparrow})} =
    \sum_{k = 2}^{\ell} p_k =
    3 H^{(1)}_{\ell} - 4 + 3/\ell$.
  \item $\Var{\mathrm{length}(\gamma^{\uparrow})} =
    \sum_{k = 2}^{\ell} p_k (1 - p_k) =
    3 H^{(1)}_\ell + 20 - 17 H^{(2)}_{\ell} - 7/\ell + 1/\ell^2$.
  \item $\sum_{k = 2}^{\ell} p_k^2 =
    -17 H^{(2)}_\ell + 24 - 10/\ell + 1/\ell^2 = O(1)$. \qedhere
\end{itemize}
\end{proof}

\subsection{An alternative proof of Theorem~\ref{thmRTCNRandPaths}}

In this section, we give another 
proof of Theorem~\ref{thmRTCNRandPaths}. This proof is
less direct than the previous one, but it provides an alternative intuition as
to where the Poisson distribution, the $\log \ell$ order of magnitude
and the factors 2 and 3 come from.

Because writing down this proof formally would require introducing additional
notation, and because we already have a formal proof, we allow ourselves to
present it as a heuristic. In what follows, the symbol
``$\approx$'' will be used loosely to indicate that two quantities are
approximately equal.

Let us start with $\gamma^\downarrow$. 
Slowing-down time in Proposition~\ref{propForwardAlgoUnif},
consider the uniform decorated RTCN with $\ell$ leaves $\nu^\circ$ obtained by:
\begin{enumerate}
  \item Starting from one lineage. 
  \item Conditional on there being $k$ lineages, letting:
    \begin{itemize}
      \item each lineage branch at rate~$1/k$;
      \item each ordered pair of lineages hybridize at rate~$1/k$.
    \end{itemize}
  \item Stopping upon reaching $\ell$ lineages.
\end{enumerate}
Note that in this construction a branching event is viewed as
the production of a new particle by another, rather than as the splitting
of a particle into two new particles. Thus, we can consider
the path $\widetilde{\gamma}^\downarrow$
obtained by always
following the edge corresponding to the lineage of the first particle,
as illustrated in Figure~\ref{figRTCNRandPathGammaDown}.

\begin{figure}[h!]
  \centering
  \captionsetup{width=0.5\linewidth}
  \includegraphics[width=0.5\linewidth]{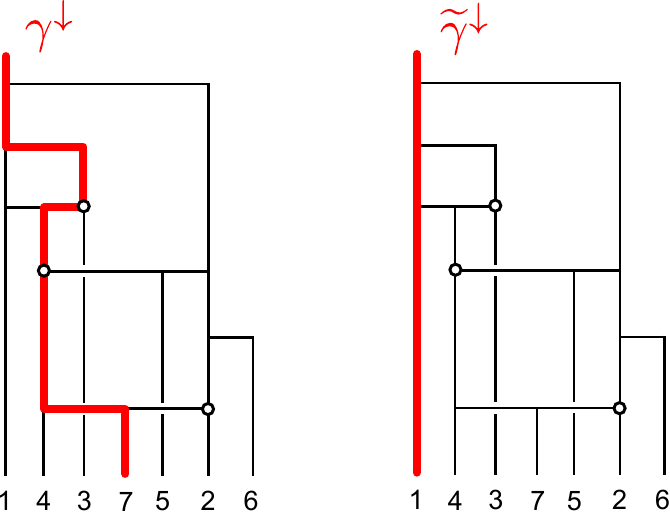} 
  \caption{On the left, the path $\gamma^\downarrow$ and on the right
  the path $\tilde{\gamma}^\downarrow$ corresponding to the lineage of the
  first particle.}  
\label{figRTCNRandPathGammaDown}
\end{figure}

From the forward-time
joint construction of $\nu^\circ$ and $\gamma^\downarrow$, we
see that the distribution of $\gamma^\downarrow$ does not depend on which
lineage it chooses to follow, as long as this choice only depends on the
past of the process. As a result,
\[
  \mathrm{length}(\gamma^\downarrow) \;\overset{d}{=}\;
  \mathrm{length}(\widetilde{\gamma}^\downarrow) \, .
\]
Now, $\mathrm{length}(\widetilde{\gamma}^\downarrow)$ is simply the number
of events affecting the lineage of the first particle, and
the rate at which each given lineage is affected by events is
$(k-1)\frac{2}{k} + \frac{1}{k} \approx 2$. Therefore, conditional on
the time $T$ it takes for the process to reach $\ell$ lineages,
\[
  \mathrm{length}(\widetilde{\gamma}^\downarrow) \approx
  \mathrm{Poisson}(2T)\,.
\]
Finally, the total number of lineages increases by $1$ at rate
$k \cdot \frac{1}{k} + k (k - 1) \cdot \frac{1}{k} = k$ and therefore
follows a Yule process $(Y(t),\, t \geq 0)$.
Since as $t \to \infty$, $Y(t) e^{-t} \to W$ almost surely, where $W$ is
a random variable (namely, an exponential variable with parameter~1),
we see that the random time $T$ it takes for the process to reach $\ell$
lineages is asymptotically
\[
  T \approx \log\ell - \log W \,.
\]
Putting the pieces together, we see that
$\mathrm{length}(\gamma^\downarrow) \approx \mathrm{Poisson}(2 \log \ell)$.

Let us now give a similar, forward-in-time construction of $\gamma^\uparrow$
where it can be identified with the lineage of the first particle. For this,
we need to ``straighten'' $\gamma^\uparrow$ thanks to a set of deterministic
rules telling us how to fix each bend, as illustrated in
Figure~\ref{figRTCNStraightenGamma}.
This yields a
$\widetilde{\nu}^\circ = f(\nu^\circ, \gamma^\uparrow)$ in which
$\widetilde{\gamma}^\uparrow$, the image of $\gamma^\uparrow$, is the lineage
of the first particle.

\begin{figure}[h!]
  \centering
  \captionsetup{width=0.95\linewidth}
  \includegraphics[width=1\linewidth]{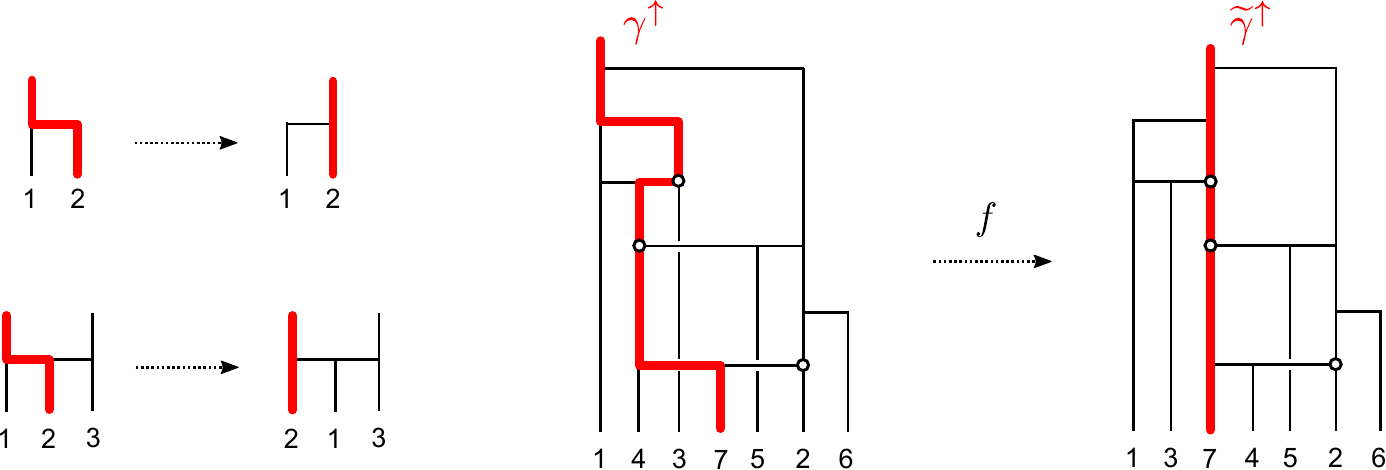}
  \caption{Illustration of the deterministic procedure used to
  ``straighten'' the path $\gamma^\uparrow$ in order to make it coincide
  with the lineage of the first particle. Left: the local modifications that
  are made each time a branching or a reticulation is encountered; these
  essentially consist of swapping lineages. Right: an example of application
  of the procedure to a RTCN.}  
\label{figRTCNStraightenGamma}
\end{figure}

Having done this,
$\mathrm{length}(\gamma^\uparrow) = \mathrm{length}(\widetilde{\gamma}^\uparrow)$
and, conditional on $\widetilde{\nu}^\circ$, the length
of~$\widetilde{\gamma}^\uparrow$ is the number of events affecting the lineage
of the first particle. 

Now, $\widetilde{\nu}^\circ$ has the same distribution as the
decorated RTCN $\hat{\nu}^\circ$ generated by the following continuous-time
Markov chain:
\begin{itemize}
  \item Each lineage branches at rate $1/k$, except for that of the first
    particle, which branches at rate $2/k$.
  \item Each ordered pair of lineages hybridizes at rate $1/k$, except
    for pairs involving the lineage of the first particle, which hybridize
    at rate $3/(2k)$.
\end{itemize}
This equality in distribution is proved by writing down explicitly the laws
$\widetilde{\nu}^\circ$ and~$\hat{\nu}^\circ$, something that requires introducing
notation to give a formal description of~$\widetilde{\nu}^\circ$.
However, to see where the factors $2$ and $3/2$ for the lineage of the
first particle come from in the construction of
$\hat{\nu}^\circ$, it suffices to note that
when going from $k - 1$ to $k$ lineages
the probability that the next event involves the lineage of the first
particle~is
\[
  \frac{3(k - 2) + 2}{3(k - 2) + 2 + (k - 2)(k - 3) + (k - 2)} \;=\;
  \frac{3k - 4}{k(k - 1)} \,, 
\]
and so is indeed the same as the probability the lineage containing
$\gamma^\uparrow$ is involved when going from $k$ to $k - 1$ lineages in the
joint backward-time construction of $\nu^\circ$ and $\gamma^\uparrow$ given in
Section~\ref{secRTCNLengthUp}.

In the construction of $\hat{\nu}^\circ$, 
conditional on there being $k$ lineages, events affect
the lineage of the first particle at rate
$2(k - 1) \cdot \frac{3}{2k} + \frac{2}{k} \approx 3$
so that letting $\hat{T}$ denote the random time it takes for the
process to reach $\ell$ lineages, $\mathrm{length}(\gamma^\uparrow) \approx
\mathrm{Poisson}(3 \hat{T})$. Finally, since
the total number of lineages increases by 1 at rate
$k + 1$ when there are $k$ lineages, the total number of lineages is distributed
as $\hat{Y}(t) - 1$, where $(\hat{Y}(t),\, t \geq 0)$ is a Yule
process started from $2$, so that $\hat{Y}(t)e^{-t} \to W + W'$, where
$W$ and $W'$ are independent exponential variables with parameter 1. Therefore,
$\hat{T} \approx \log \ell$ and we recover
$\mathrm{length}(\gamma^\uparrow)\approx\mathrm{Poisson}(3\log\ell)$.

%]]]

\section{Number of lineages in the ancestry of a leaf} \label{secRTCNAncestryLeaf}

Let us start by giving a formal definition of the process
counting the number of lineages in the ancestry of a leaf that was
described in Section~\ref{secRTCNMainResults}.

\subsection{%
\texorpdfstring{Definition and characterization of $X^{(\ell)}$}%
{Definition and characterization of X}%
} \label{appRTCNVariance}

\begin{definition}
The \emph{ancestry} of a vertex $u$ of a RTCN $\nu$ is the subgraph
$\Restriction{\nu}{u}$ consisting of all paths going from the root of
$\nu$ to $u$.
\end{definition}

\begin{definition} \label{defLineages}
The \emph{number of lineages} of a subgraph $\mu$ of a RTCN $\nu$
is the process $(X_k, \, 0 \leq k \leq \ell - 2)$ defined by
\[
  X_k \;=\;
  \#\Set*[\big]{e \in \mu \suchthat
    \text{the edge }e \text{ is alive between $U_k$ and $U_{k + 1}$}} \, , 
\]
where $U_1 \succ \cdots \succ U_{\ell - 1}$ are the events of $\nu$,
taken in inverse chronological order, 
with the same convention as
in Definition~\ref{defAliveEdges} of alive edges that $U_0 = \partial V$.
\end{definition}

In the rest of this section, we study the number of lineages
in the ancestry of a uniformly chosen leaf of a uniform RTCN with $\ell$ labeled
leaves, and denote it by~$X^{(\ell)}$. 
See Figure~\ref{figRTCNDefAncestry} in Section~\ref{secRTCNMainResults} for an
illustration.  We also study the embedded process $\tilde{X}^{(\ell)}$ defined
by $\tilde{X}_i = X_{k_i}$, where $k_0 = 0$ and, for~$i \geq 1$,
$k_i = \inf\Set*{k > k_{i - 1} \suchthat X_k \neq X_{k_{i - 1}}}$.

Let us start by characterizing the law of $X^{(\ell)}$.

\begin{proposition}
The process $X_k^{(\ell)}$ is the Markov chain characterized by
$X^{(\ell)}_0 = 1$ and the transition probabilities:
\begin{itemize}
  \item $\displaystyle\Prob{X^{(\ell)}_{k + 1} = x + 1 \given X^{(\ell)}_k = x} =
    \frac{x(\ell - k - x)(\ell - k - x - 1)}{(\ell - k)(\ell - k - 1)^2}$
  \item $\displaystyle\Prob{X^{(\ell)}_{k + 1} = x - 1 \given X^{(\ell)}_k = x} =
    \frac{x(x - 1)^2}{(\ell - k)(\ell - k - 1)^2}$\\
  \item $\displaystyle\Prob{X^{(\ell)}_{k + 1} = x \given X^{(\ell)}_k = x} =
    1 -\Prob{X_{k + 1} = x \pm 1 \given X^{(\ell)}_k = x}$
\end{itemize}
\end{proposition}

\begin{proof}
The proof relies on the
backward construction of a uniform RTCN and a bit of book-keeping to see
how the $(k + 1)$-th event, which takes us from $\ell - k$ lineages to
$\ell - k - 1$ lineages, affects the number of lineages in the ancestry of
a leaf. Recall that in the backward construction there are
$(\ell - k)(\ell - k - 1)^2/2$ possibilities for
the $(k + 1)$-th event. Let us refer to the lineages in the ancestry of
the focal leaf as marked lineages.  Conditional on $X^{(\ell)}_k = x$, there
are $x$ marked lineages and $\ell - k - x$ unmarked lineages and thus there
are:
\begin{itemize}
  \item $x(x - 1)/2$ possible coalescences between marked lineages. These
    decrease the number of lineages by 1.
  \item $x(x - 1)(x - 2) / 2$ reticulations involving only marked lineages. These
    also decrease the number of lineages by 1.
  \item $x (\ell - k - x)(\ell - k - x - 1)/2$ possible reticulations where
    the hybrid is a marked lineage and the other two lineages are
    unmarked. These increase the number of marked lineages by 1.
\end{itemize}
Other types of events
(coalescences between two unmarked lineages, coalescences between
a marked and an unmarked lineage, etc...) leave the number of marked
lineages unchanged, as illustrated in Figure~\ref{figRTCNProofAncestry01}.
Since the event is chosen uniformly among all possibilities, this
concludes the proof.
\end{proof}

\begin{figure}[h!]
  \centering
  \captionsetup{width=0.9\linewidth}
  \includegraphics[width=0.85\linewidth]{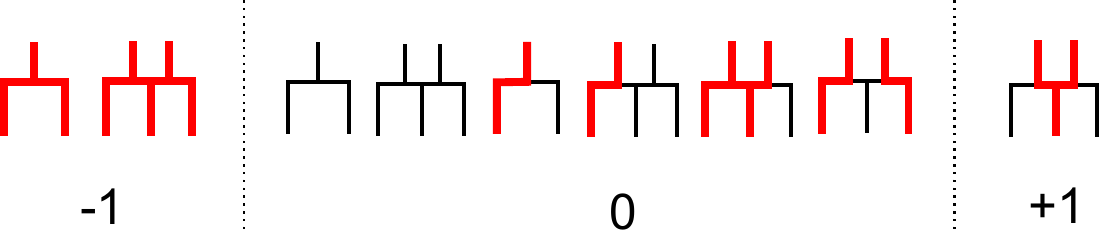}
  \caption{List of all the possible types of events and their effect on
  the number of marked lineages.}  
\label{figRTCNProofAncestry01}
\end{figure}

\subsection{Simulations} \label{secRTCNSimus}

\enlargethispage{2ex}
In this section, we present simulations supporting the conjectures about
$X^{(\ell)}$ and $\tilde{X}^{(\ell)}$ made in Section~\ref{secRTCNMainResults}
and outline some ideas to approach these conjectures. Let us start by
looking at some individual trajectories of these processes, for increasing
values of $\ell$. As can be seen in
Figure~\ref{figRTCNIndivTrajs}\hyperref[figRTCNIndivTrajs]{.A}, most of the
interesting behavior of $X^{(\ell)}$ seems to happen very close to the root so
that to obtain a non-degenerate scaling-limit, we need to focus on a small
window of time, for instance by considering
$X^{(\ell)}_k$ for 
$k = \lfloor \ell - M\sqrt{\ell}(1 - t)\rfloor$ and
$t \in \ClosedInterval{0, 1}$.

\begin{figure}[h!]
  \centering
  \captionsetup{width=0.85\linewidth}
  \includegraphics[width=0.85\linewidth]{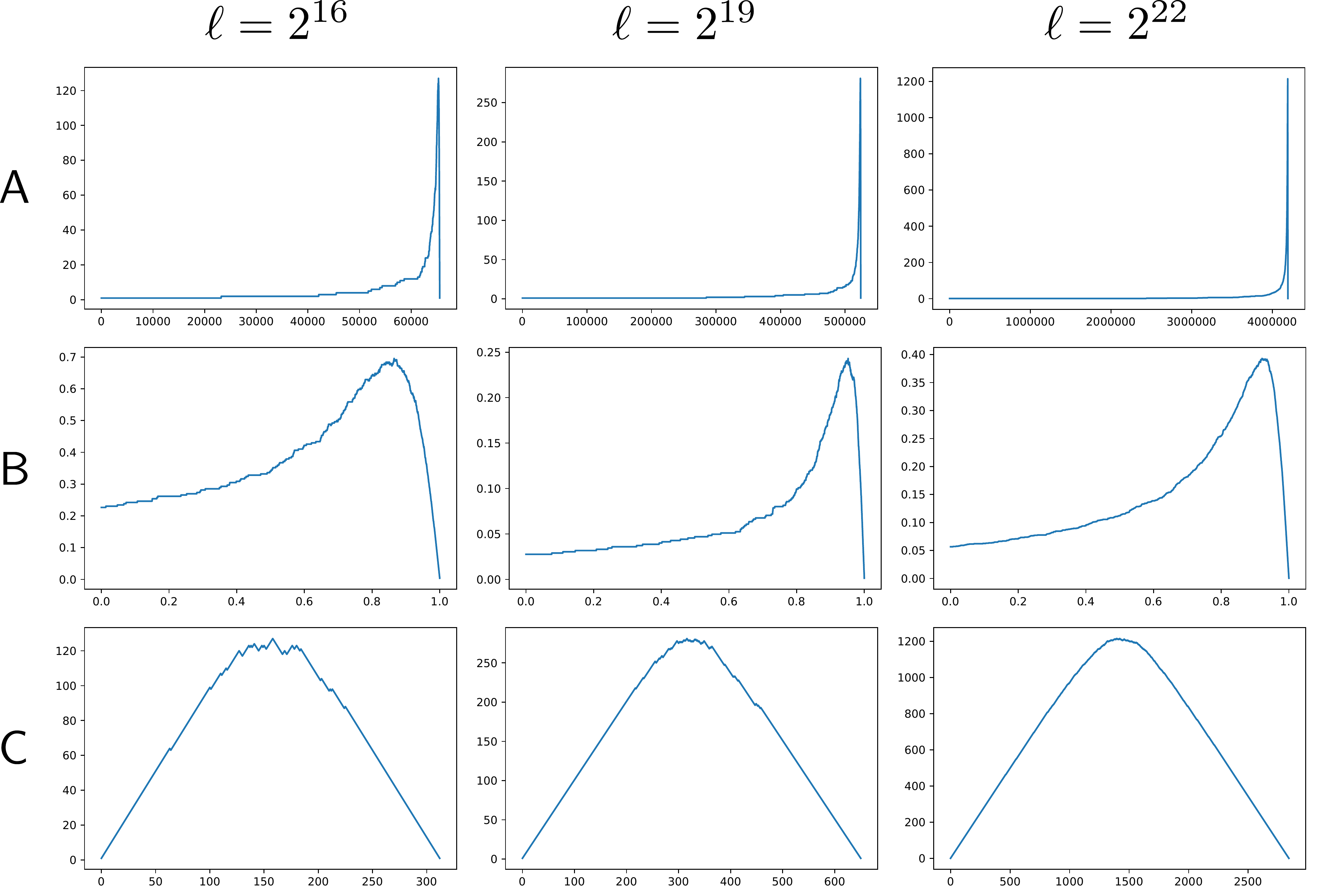}
  \caption{Individual trajectories of the processes described in the main
  text, for different values of $\ell$. $\mathsf{A}$, the process
  $X^{(\ell)}$, $\mathsf{B}$ the process
  $t \mapsto \frac{1}{\sqrt{\ell}} X^{(\ell)}_{\lfloor \ell - M\sqrt{\ell}(1 - t)\rfloor}$
  and $\mathsf{C}$, the process $\tilde{X}^{(\ell)}$.}
\label{figRTCNIndivTrajs}
\end{figure}

Even though the trajectories represented in Figure~\ref{figRTCNIndivTrajs}
seem to become smooth as $\ell \to \infty$, they
do not become deterministic, as made apparent by Figure~\ref{figRTCNHist}, where
the distributions of some statistics of $X^{(\ell)}$ are given. In particular,
these simulations suggest that the relevant scaling limit for
$X^{\ell}$ is indeed 
$\frac{1}{\sqrt{\ell}}
X^{(\ell)}_{\lfloor \ell - M\sqrt{\ell}(1 - t) \rfloor}$, and
support Conjecture~\ref{conjRTCNMain}.

\begin{figure}[h!]
  \centering
  \captionsetup{width=0.85\linewidth}
  \includegraphics[width=0.85\linewidth]{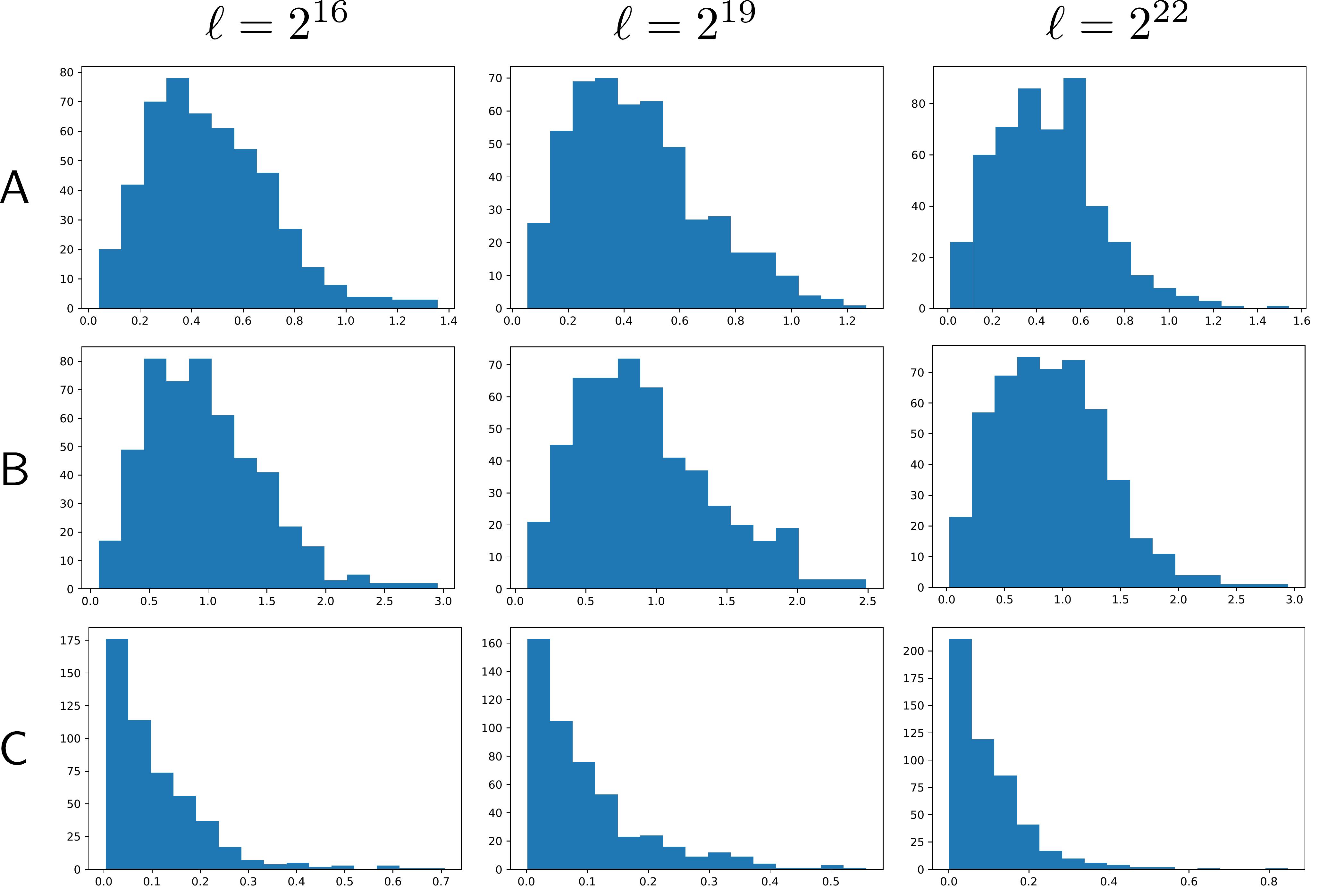}
  \caption{Distribution of some statistics of $X^{(\ell)}$ for $M = 10$
  and 500 trajectories.
  $\mathsf{A}$,
  $\frac{1}{\sqrt{\ell}}
  \max\Set*{ X^{(\ell)}_k \suchthat 0 \leq k \leq \ell - 2}$; $\mathsf{B}$, 
  $\frac{1}{\sqrt{\ell}}\big(\ell - \argmax\Set*{ X^{(\ell)}_k \suchthat 0 \leq k \leq \ell - 2}\big)$ and $\mathsf{C}$, 
  $\frac{1}{\sqrt{\ell}} X^{(\ell)}_{\lfloor \ell - M\sqrt{\ell} \rfloor}$}  
\label{figRTCNHist}
\end{figure}

Our idea to approach the study of the process $X^{(\ell)}$
is to separate it into two phases:
\begin{enumerate}
  \item A slow, stochastic phase, where up to time $k = \lfloor \ell -
    M\sqrt{\ell}\rfloor$ the process $X^{(\ell)}_k$ remains relatively
    small and highly stochastic
  \item A fast, deterministic
    phase, during which the internal dynamics of the
    rescaled process become deterministic, but retain a trace of the
    stochasticity of the first phase in the form of random initial conditions.
\end{enumerate}

\subsection{The stochastic phase}

In this section, to avoid clutter we will sometimes drop the
superscript in $X^{(\ell)}$.

Let us start by considering the process
$Z = (Z_k, 0\leq k \leq \ell - 2)$ characterized
by $Z_0 = 1$, $Z_{k + 1} - Z_k \in \{0, 1\}$ and
\[
  \Prob{Z_{k + 1} = z + 1 \given Z_k = z}
  \;=\; \frac{z}{\ell - k - 1} \,.
\]
Comparing the transition probabilities of $Z$ with those of
$X$, we see that we can couple these two processes in such a way that
$X_k \leq Z_k$ for all $k$. Let us now compute the first
moments of $Z$.
\begin{proposition} \label{propRTCNZk}
~
\begin{mathlist}
\item $\displaystyle\Expec{Z_k} = \frac{\ell}{\ell - k}$
\item $\displaystyle\Expec{Z_k^2} =
  \frac{\ell(\ell + k + 1)}{(\ell - k)(\ell - k + 1)}$
\end{mathlist}  
\end{proposition}

\begin{proof}
We have
\[
  \Expec{Z_{k + 1}} = \mleft(1 + \frac{1}{\ell- k - 1}\mright)\,\Expec{Z_k}\,.
\]
Using that $\Expec{Z_0} = 1$, we get
\[
  \Expec{Z_k} \;=\;
  \prod_{i = 0}^{k - 1} \mleft(1 + \frac{1}{\ell- i - 1}\mright)
  \;=\;\frac{\ell}{\ell - k} \, , 
\]
proving (i). For (ii), we note that
\[
  \Expec{Z_{k + 1}^2 \given Z_k} \;=\;
  \frac{Z_k}{\ell - k - 1} + \mleft(\frac{2}{\ell - k - 1} + 1\mright) Z_k^2\, .
\]
Taking expectation and substituting $\Expec{Z_k}$ by $\ell / (\ell - k)$, we
get
\[
  \Expec{Z_{k + 1}^2} \;=\;
  \frac{\ell}{(\ell - k)(\ell - k - 1)} \;+\;
  \frac{\ell - k + 1}{\ell - k - 1} \, 
  \Expec{Z_{k}^2}\,.
\]
Multiplying both sides by $(\ell - k)(\ell - k - 1)$,
\[
  (\ell - k)(\ell - (k + 1))\, \Expec{Z_{k + 1}^2} \;=\;
  \ell \;+\;
  (\ell - (k - 1))(\ell - k) \, 
  \Expec{Z_{k}^2}
\]
and therefore
\[
  \Expec{Z_k^2} \;=\;
  \frac{\ell(\ell + k + 1)}{(\ell - k)(\ell - k + 1)} \, , 
\]
concluding the proof.
\end{proof}

\begin{proposition} \label{propRTCNExpecAncesLineage}
For all $\ell \geq 2$ and $k$ such that $0 \leq k \leq \ell - 2$,
\[
  \frac{\ell}{\ell - k + 1}
  \mleft(1 - \frac{2 k}{(\ell - k)(\ell - k - 1)}\mright) \;\leq\;
  \Expec{X_k} \;\leq\;
  \frac{\ell}{\ell - k}\ .
\]
As a result, for $M > 2$ and all $\ell$ large enough, for all $\epsilon > 0$,
\[
  (1 - \epsilon) \frac{1}{M} \mleft(1 - \frac{2}{M^2}\mright) \;\leq\;
  \Expec{\tfrac{1}{\sqrt{\ell}}
         X_{\lfloor\ell - M \sqrt{\ell}\rfloor}^{(\ell)}} \;\leq\;
  \frac{1}{M} \,, 
\]
so that the sequence of random variables
$\frac{1}{\sqrt{\ell}} X_{\lfloor\ell - M \sqrt{\ell}\rfloor}^{(\ell)}$
is tight and bounded away from 0 in $L^1$.
\end{proposition}

\begin{proof}
The upper bound on $\Expec{X_k}$ follows immediately from
$X_k \leq Z_k$ and Proposition~\ref{propRTCNZk}.
For the lower bound, let us denote by
$\mathcal{F}_k = \sigma(X_0, \ldots, X_{k}; Z_0, \ldots, Z_k)$ the
filtration generated by $Z$ and $X$. Then,
\begin{align*}
  \Expec{X_{k + 1} \given \mathcal{F}_k}
  \;&= X_k +
     \frac{X_k(\ell - k - X_k)(\ell - k - X_k - 1)}{(\ell - k)(\ell - k - 1)^2} - 
      \frac{X_k(X_k - 1)^2}{(\ell - k)(\ell - k - 1)^2}\\[1ex]
  \;&=\; X_k\, \mleft(1 
      + \frac{(\ell - k)(\ell - k - 1) - 1}{(\ell - k)(\ell - k - 1)^2}
      - \frac{2(\ell - k) - 3}{(\ell - k)(\ell - k - 1)^2} \, X_k\mright) \\[1ex]
    \;&\geq\; X_k\, \mleft(1
      + \frac{1}{\ell - k}\mright) - \frac{2\, Z_k^2}{(\ell - k - 1)^2} 
\end{align*}
Taking expectations, 
\[
  \Expec{X_{k + 1}}
    \;\geq\; \mleft(1 + \frac{1}{\ell - k}\mright)\, \Expec{X_k}
    - \frac{2\,}{(\ell - k - 1)^2} \Expec{Z_k^2} \,, 
\]
and using Proposition~\ref{propRTCNZk} we get
\[
  \Expec{X_{k + 1}}
    \;\geq\; \mleft(1 + \frac{1}{\ell - k}\mright)\, \Expec{X_k}
    - \frac{2\,\ell(\ell + k + 1)}{(\ell - k + 1)(\ell - k)(\ell - k - 1)^2}\, .
\]
Multiplying both side of the inequality by $(\ell - k)$, we get
\[
  (\ell - k)\, \Expec{X_{k + 1}}
    \;\geq\; (\ell - (k - 1))\, \Expec{X_k}
    - \frac{2\,\ell(\ell + k + 1)}{(\ell - k + 1)(\ell - k - 1)^2}\, .
\]
and as a result,
\[
  (\ell - (k - 1))\, \Expec{X_{k}}
    \;\geq\; \ell + 1
    - 2 \ell \sum_{i = 0}^{k - 1} 
    \frac{\ell + i + 1}{(\ell - i + 1)(\ell - i - 1)^2}\, .
\]
Using the fact that
\[
  \frac{\ell + i + 1}{(\ell - i + 1)(\ell - i - 1)^2} \;\leq\;
  \frac{\ell + i}{(\ell - i)(\ell - i - 1)(\ell - i - 2)}
\]
and that
\[
  \sum_{i = 0}^{k - 1}
  \frac{\ell + i}{(\ell - i)(\ell - i - 1)(\ell - i - 2)}
  \;=\;
  \frac{k}{(\ell - k)(\ell - k - 1)} \,,
\]
we get
\[
  \Expec{X_k} \;\geq\;
  \frac{\ell}{\ell - k + 1}
  \mleft(1 - \frac{2 k}{(\ell - k)(\ell - k - 1)}\mright) \, .
\]
Finally, by taking $k = \lfloor \ell - M\sqrt{\ell}\rfloor$ in these
inequalities we get
\[
  \frac{\ell}{\sqrt{\ell}(M\sqrt{\ell} + 2)}
  \mleft(1 - \frac{2(\ell - M\sqrt{\ell})}{M\sqrt{\ell}(M\sqrt{\ell} - 1)}\mright)
  \;\leq\;
  \frac{1}{\sqrt{\ell}}\Expec{X_{\lfloor \ell - M \sqrt{\ell}\rfloor}} \;\leq\;
  \frac{1}{M}
\]
where the term on the left-hand side of the inequality goes to
$\frac{1}{M}(1 - \frac{2}{M^2})$ as $\ell \to \infty$.
Finally, the tightness follows from Markov's inequality.
\end{proof}

Proposition~\ref{propRTCNExpecAncesLineage} suggests the following
conjecture.

\begin{conjecture} \label{conjRTCNMain}
The sequence of random variables
$\frac{1}{\sqrt{\ell}} X^{(\ell)}_{\lfloor\ell - M \sqrt{\ell}\rfloor}$
converges in distribution to a positive random variable $W_M$.
\end{conjecture}

One problem that arises in attempting to prove Conjecture~\ref{conjRTCNMain} it
is that as soon as we get in the regime
$k = \lfloor \ell - M\sqrt{\ell}(1 - t)\rfloor$, the random variables
$Z^{(\ell)}$ and $X^{(\ell)}$ start to
differ significantly and therefore the coupling is not so useful.

A very natural idea would be to use the backward-time construction to couple
$X^{(\ell)}$ and $X^{(\ell + 1)}_k$, but a difficulty with this approach is that
this coupling lacks continuity in the sense that, with
probability~$\Theta(1/\ell)$, $X^{(\ell)}_k$ and $X^{(\ell + 1)}$ will differ by
a factor $\Theta(\ell)$ for $k = \lfloor \ell - M \sqrt{\ell}\rfloor$.

\subsection{The deterministic phase}

\begin{proposition} \label{propRTCNConvergeDeter}
If Conjecture~\ref{conjRTCNMain} holds, that is, if
\[
  \tfrac{1}{\sqrt{\ell}} X^{(\ell)}_{\lfloor\ell - M \sqrt{\ell}\rfloor}
  \tendsto[d]{\ell \to \infty} W_M  > 0
\]
then, for all $\epsilon$ such that $0 < \epsilon < 1$, as $\ell \to \infty$,
\[
  \Big(\tfrac{1}{M\sqrt{\ell}}
  X^{(\ell)}_{\lfloor \ell -M \sqrt{\ell}(1 - t)\rfloor}, \;
  t \in \ClosedInterval{0, 1 - \epsilon}\Big)
  \;\implies\;
  \big(y(t, C_M),\; t \in \ClosedInterval{0, 1 - \epsilon}\big)
\]
where $\implies$ denotes convergence in distribution in the Skorokhod space,
\[
  y(t, C_M) \;=\; \frac{1 - t}{C_M \cdot(1 - t)^2 + 1}
\]
and $C_M = M / W_M - 1$.
\end{proposition}

\begin{proof}
Let us write for convenience $M_\ell \defas M \sqrt{\ell}$ and
\[
  \mathscr{T}_\ell =
  \tfrac{1}{M_\ell} \Set*[\big]{0, \ldots, \ell - 2 - \lfloor \ell - M_\ell\rfloor}\,.
\]
Define the Markov chain $(Y^{(\ell)}_t, \, t \in \mathscr{T}_\ell)$ taking
values in $\frac{1}{M_\ell}\N$ by
\[
  Y^{(\ell)}_t  \;=\;
  \tfrac{1}{M_\ell} X^{(\ell)}_{\lfloor \ell - M_\ell \rfloor + tM_\ell} \,.
\]
The Markov chain $Y^{(\ell)}$  has infinitesimal mean
\begin{align*}
  b^{(\ell)}(y, t)
  \;&=\; M_\ell\, \Expec{Y_{t + 1/M_\ell}^{(\ell)} - y \given Y_{t}^{(\ell)} = y} \\[0.5ex]
  \;&=\;
  \Expec{X_{\lfloor \ell - M_\ell \rfloor + tM_\ell + 1}^{(\ell)} - yM_\ell \given
  X_{\lfloor \ell - M_\ell \rfloor + tM_\ell}^{(\ell)} = yM_\ell} \\[0.7ex]
  \;&=\;
    \frac{yM_\ell(\ell - k - yM_\ell)(\ell - k - yM_\ell - 1)}{(\ell - k)(\ell - k - 1)^2}
     - \frac{yM_\ell(yM_\ell - 1)^2}{(\ell - k)(\ell - k - 1)^2}
\end{align*}
where $k = \lfloor \ell - M_\ell \rfloor + tM_\ell$. Let us show that, for
any $R > 0$ and any $\epsilon > 0$,
\[
  b^{(\ell)}(y, t) \tendsto{\ell \to \infty}
  \frac{y(1 - t - y)^2}{(1 - t)^3} - \frac{y^3}{(1 - t)^3}\,, 
\]
uniformly in $(y, t) \in \ClosedInterval{0, R}  \times
\ClosedInterval{0, 1 - \epsilon}$. Let us write
\[
  \begin{dcases}
    b_+^{(\ell)}(y, t) = 
   \frac{yM_\ell(\ell - k - yM_\ell)(\ell - k - yM_\ell - 1)}{(\ell - k)(\ell - k - 1)^2}
    \\[0.5ex]
    b_-^{(\ell)}(y, t) =
    \frac{yM_\ell(yM_\ell - 1)^2}{(\ell - k)(\ell - k - 1)^2}
  \end{dcases}
\]
Using that
$(1 - t)M_\ell - 1 \leq \ell - k \leq (1 - t) M_\ell$, we get
\[
  \frac{y(1 - t - y - 2/M_\ell)^2}{(1 - t)^3}
  \;\leq\; b_+^{(\ell)}(y, t) \;\leq\;
  \frac{y(1 - t - y)^2}{(1 - t - 2/M_\ell)^3}
\]
As a result, 
\begin{align*}
  b_+^{(\ell)}(y, t) -
  \frac{y(1 - t - y)^2}{(1 - t)^3}
  \;&\geq\; \frac{y}{(1 - t)^3}\mleft(- \frac{4}{M_\ell}(1 - t -y) +
  \frac{4}{M_\ell^2}\mright) \\[0.5ex]
  \;&\geq\;
  -\frac{4R}{\epsilon^3 M_\ell} + O\big(1/M_\ell^2\big)
\end{align*}
Similarly,
\begin{align*}
  b_+^{(\ell)}(y, t) -
  \frac{y(1 - t - y)^2}{(1 - t)^3}
  \;&\leq\;
  \frac{y(1 - t - y)^2}{(1 - t)^3(1 - t - 2/M_\ell)^3}\mleft(
  (1 - t)^3 - (1 - t - 2/M_\ell)^3
  \mright) \\[0.5ex]
  \;&\leq\;
  \frac{R(1 + R)^2}{\epsilon^3(\epsilon - 2/M_\ell)^3}\mleft(
  \frac{6}{M_\ell} + O\big(1/M_\ell^2\big) \mright) \, .
\end{align*}
This proves the uniform convergence of
$b_+^{(\ell)}(y, t)$. The uniform convergence of
$b_-^{(\ell)}(y, t)$ is treated similarly.

Now, since $X^{(\ell)}$ has jumps of size one,
$Y^{(\ell)}$ has infinitesimal variance
\begin{align*}
  a^{(\ell)}(y, t) \;&=\;
  M_\ell\, \Expec{\mleft(Y_{t + 1/M_\ell}^{(\ell)} - y\mright)^2 \given 
  Y_{t}^{(\ell)} = y}  \\
  \;&=\; \frac{1}{M_\ell} \mleft(b_+^{(\ell)}(y, t) +
  b_-^{(\ell)}(y, t)\mright) \,, 
\end{align*}
which goes to zero uniformly in $(y, t) \in \ClosedInterval{0, R}  \times
\ClosedInterval{0, 1 - \epsilon}$. Assuming that
\[
  \tfrac{1}{M_\ell} X^{(\ell)}_{\lfloor\ell - M_\ell \rfloor}
  \tendsto[d]{\ell \to \infty} W_M / M \,, 
\]
the convergence of the piecewise constant interpolation of
$Y^{(\ell)}$ to the solution of the Cauchy problem
\[
  \begin{dcases}
  \frac{dy}{dt} \;=\;
  \frac{y(1 - t - y)^2}{(1 - t)^3} - \frac{y^3}{(1 - t)^3}\\
    y(0) = W_M / M
  \end{dcases}
\]
follows from Corollary 4.2 of \cite{EthierKurtz} (see for instance
Chapter~8.7 of \cite{durrett1996stochastic} for a more practical introduction).
Note that in this references, the results are stated for time-homogeneous
Markov chains. However, they are easily adapted to time-inhomogeneous
ones by extending the state space with time in order to obtain a time-homogeneous
process.

Finally, the function given in the Proposition is then readily checked to be
the unique solution of that Cauchy problem, concluding the proof.
\end{proof}
To close this section, let us mention briefly that an idea to
study the embedded process $\tilde{X}^{(\ell)}$ is to
introduce the process $S^{(\ell)}$ that counts
the jumps of $X^{(\ell)}$. Indeed, with this process,
\[
  \tilde{X}^{(\ell)}_{i} = X^{(\ell)}_{(S^{(\ell)})^{-1}(i)} \, .
\]
As a result, proving the convergence
\[
  \tfrac{1}{M_\ell}
  \mleft(X^{(\ell)}_{\lfloor\ell - M_\ell \rfloor + tM_\ell}, \;
  S^{(\ell)}_{\lfloor\ell - M_\ell \rfloor + tM_\ell}
  \mright) \implies (y(t), s(t))
\]
would show that
\[
  \tfrac{1}{M_\ell}
  \tilde{X}^{(\ell)}_{\lfloor \ell - M_\ell(1 - t)\rfloor}
  \implies y(s^{-1}(t))  \,,
\]
and we might be able to take $M \to \infty$ to study the convergence of
$\tfrac{1}{\sqrt{\ell}} \tilde{X}^{(\ell)}_{\lfloor \ell t\rfloor}$.

\section*{Acknowledgements}

FB\ and AL\ thank the
Center for Interdisciplinary Research in Biology (CIRB) for funding.
MS\ thanks the New Zealand Marsden Fund (UOC1709) for funding.
The authors thank
Régine Marchand, Michael Fuchs and two anonymous referees for
helpful corrections and suggestions. FB\ thanks
Jean-Jil Duchamps and Peter Czuppon for
discussions about Conjecture~\ref{conjRTCNMain}.

\appendix

\bibliographystyle{abbrv}
\bibliography{biblio}

\newpage

\section{Lemmas used in Section~\ref{secRTCNStats}}

In this section, we recall the proof of two elementary lemmas that were used
to study the number of cherries and of tridents in
Section~\ref{secRTCNStats}.

\begin{lemma} \label{lemmaRTCNRecur}
Let $(u_\ell)$ be a sequence satisfying the recursion
\[
  u_{\ell + 1} = a_\ell\, u_\ell + b_\ell
\]
and let $i$ be such that $\forall \ell \geq i$, $a_\ell \neq 0$. Then,
\[
  \forall \ell \geq i, \quad
  u_\ell = \mleft(u_i + \sum_{k = i}^{\ell - 1}
  \frac{b_k}{\prod^k_{j = i}a_j} \mright)\prod_{k = i}^{\ell - 1}a_k 
\]
\end{lemma}

\begin{proof}
For $k \geq i$, since $a_k \neq 0$ we have
\[
  \frac{u_{k + 1}}{\prod_{j = i}^{k} a_j} \;-\;
  \frac{u_{k}}{\prod_{j = i}^{k - 1} a_j} \;=\;
  \frac{b_k}{\prod_{j = i}^{k} a_j} \,,
\]
where the empty product is one.
As a result, for all $\ell \geq i$,
\[
  \frac{u_{\ell}}{\prod_{j = i}^{\ell - 1} a_j} \;-\; u_i
  \;=\;
  \sum_{k = i}^{\ell - 1}
  \mleft(\frac{u_{k + 1}}{\prod_{j = i}^{k} a_j} -
  \frac{u_{k}}{\prod_{j = i}^{k - 1} a_j}\mright) \;=\;
  \sum_{k = i}^{\ell - 1}
  \frac{b_k}{\prod_{j = i}^{k} a_j} \,,
\]
and the proof is over.
\end{proof}

\begin{lemma} \label{lemmaStupid}
Let $(v_\ell)$ be such that $v_\ell \sim \alpha\, \ell^{\, p}$, where $p \geq 0$
and $\alpha \neq 0$.
Then, 
\[
  \sum_{k = i}^{\ell - 1} v_k \sim \frac{\alpha}{p + 1} \ell^{\,p + 1} \,.
\]
\end{lemma}
\begin{proof}
Let $\epsilon_\ell \to 0$ be such that
$v_\ell = \alpha\, \ell^{\, p} \, +\,  \epsilon_\ell\, \ell^{\,p}$.
Then,
\[
  \sum_{k = i}^{\ell - 1} v_k \;=\;
  \alpha \sum_{k = i}^{\ell - 1} k^{\, p} +
  \sum_{k = i}^{\ell - 1} \epsilon_k\, k^{\, p}
\]
Comparison with an integral shows that $\sum_{k = i}^{\ell - 1} k^{\, p} \sim
\ell^{\,p + 1} / (p + 1)$, so to finish the proof we simply have to show that
$\sum_{k = i}^{\ell - 1} \epsilon_k\, k^{\, p} = o(\ell^{\, p + 1})$.
Since $\epsilon_\ell \to 0$, for all $\eta > 0$ there
exists $j_\eta$ such that $\forall k \geq j_\eta$, 
$\Abs{\epsilon_k} < \eta$. Therefore,
\[
  \Abs{\sum_{k = i}^{\ell - 1} \epsilon_k\, k^{\, p}}
 \;<\;
  \sum_{k = i}^{j_\eta - 1} \Abs{\epsilon_k\, k^{\, p}} \;+\;
  \eta \sum_{k = j_\eta}^{\ell - 1} k^{\, p}
\]
For fixed $\eta$, the first of these sums has a fixed number of terms and
thus is bounded. The
second one is asymptotically equivalent to $\eta\, \ell^{p + 1} / (p + 1)$.
Thus, for all $\eta > 0$, for $\ell$ large enough,
\[
  \Abs{\frac{\sum_{k = i}^{\ell - 1} \epsilon_k\, k^{\, p}}{\ell^{\,p + 1}}} 
 \;<\; \eta
\]
and the proof is over.
\end{proof}

\section{%
\texorpdfstring{Variance of $\chi_\ell$}%
{Variance of the number of tridents}%
} \label{appRTCNVariance}

In this section, we prove Proposition~\ref{propRTCNVariance} by obtaining
an explicit expression for the variance of the number $\chi_\ell$ or
tridents of a uniform RTCN with $\ell$ labeled leaves.

Consider the Markov chain $(X_\ell)_{\ell \geq 2}$ defined in
Proposition~\ref{propRTCNMarkovchainRetCherries}, and let
$s_\ell = \Expec{X_\ell^2}$. From the transition probabilities of
$X_\ell$, we get
\begin{align*}
  \Expec{X_{\ell + 1}^2 \given X_\ell = k} =
  \mleft(\frac{\ell - 6}{\ell}\mright)^2 \, k^2 \;+\;
  \frac{2\ell^2 - \ell -3}{\ell^2} \, k \;+\; \frac{\ell - 1}{\ell} \, .
\end{align*}
Integrating in $k$, this yields
\begin{align*}
  s_{\ell + 1} = 
  \mleft(\frac{\ell - 6}{\ell}\mright)^2 \, s_{\ell} \;+\;
  \frac{2\ell^2 - 8\ell -3}{\ell^2} \, \mu_{\ell} \;+\; \frac{\ell - 1}{\ell} \,
\end{align*}
where, for $\ell \geq 4$, we
can substitute the expression of $\mu_{\ell}$ given in
Proposition~\ref{propRTCNExpecRetCherries}. Rearranging a bit
get that for all $\ell \geq 4$,
\[
  s_{\ell + 1} = 
  \mleft(\frac{\ell - 6}{\ell}\mright)^2 \, s_{\ell} \;+\;
  \frac{30\, \ell^5 - 185\, \ell^4 + 188\, \ell^3 + 746\, \ell^2 - 1649\, \ell + 843}
  {105\, (\ell - 1) (\ell - 2) (\ell - 3) \ell}
\]
Using Lemma~\ref{lemmaRTCNRecur} and a symbolic computation software, we
get an explicit expression for $s_\ell$, and, from there,
\begin{align*}
  \Var{\chi_\ell} &=
  \frac{{\big(59400 \, \ell^{9} - 1618650 \, \ell^{8} + O(\ell^7)\big)} \ell}{
    1576575 \, {(\ell - 1)}^{2} {(\ell - 2)}^{2} {(\ell - 3)}^{2} {(\ell - 4)}
  {(\ell - 5)} {(\ell - 6)}}
\end{align*}
from which Proposition~\ref{propRTCNVariance} follows.

\end{document}